\documentclass{amsart}

\usepackage[utf8]{inputenc}
\usepackage[T1]{fontenc}
\usepackage[english]{babel}
\usepackage{amsmath,amssymb,amsfonts,amsthm}
\usepackage[pdftex]{graphicx}
\usepackage{fourier}
\usepackage{enumerate}
\usepackage[usenames,dvipsnames]{color}

\newtheorem{thm}{Theorem}[section]
\newtheorem{lem}[thm]{Lemma}
\newtheorem{prop}[thm]{Proposition}
\newtheorem{cor}[thm]{Corollary}
\newtheorem{conj}[thm]{Conjecture}

\theoremstyle{definition}
\newtheorem{defi}[thm]{Definition}

\theoremstyle{remark}
\newtheorem{rem}[thm]{Remark}
\newcommand{\de}{\, \mathrm{d}}
\newcommand{\del}{\partial}

\newcommand{\Vol}{\operatorname{Vol}}

\newcommand{\Trig}{\operatorname{Trig}}
\newcommand{\Hyp}{\operatorname{Hyp}}

\newcommand{\N}{\mathbb{N}}
\newcommand{\Z}{\mathbb{Z}}

\newcommand{\R}{\mathbb{R}}

\renewcommand{\S}{\mathbb S}

\newcommand{\CT}{\mathcal{T}}

\newcommand {\bx}{\mathbf x}

\newcommand {\bk}{\mathbf k}

\newcommand {\bm}{\mathbf m}

\newcommand {\by}{\mathbf y}

\newcommand {\bn}{\mathbf n}

\newcommand {\balpha}{\boldsymbol \alpha}
\newcommand {\bbeta}{\boldsymbol \beta}
\newcommand {\bgamma}{\boldsymbol \gamma}

\newcommand {\bxi}{\boldsymbol \xi}

\newcommand {\btheta}{\boldsymbol \theta}
\newcommand {\krn}{\nobreak\hspace{.16667em plus .08333em}}

\newcommand{\set}[1]{\left\{ #1 \right\}}

\newcommand{\bo}\boldsymbol{}
\newcommand{\bigo}[1]{O\left( #1 \right)}
\newcommand{\smallo}[1]{o\left( #1 \right)}

\DeclareMathOperator{\arctanh}{arctanh}

\DeclareMathOperator{\arccot}{arccot}

\renewcommand{\hat}{\widehat}
\renewcommand{\tilde}{\widetilde}

\renewcommand{\phi}{\varphi}

\numberwithin{equation}{subsection}

\begin{document}

\title{The Steklov spectrum 
  of cuboids}
\author[A. Girouard]{Alexandre Girouard}
\address{D\'epartement de math\'ematiques et de statistique, Pavillon Alexeandre-Vachon, Universit\'e Laval, Qu\'ebec, QC, G1V 0A6, Canada}
\email{alexandre.girouard@mat.ulaval.ca}

\author[J. Lagac\'e]{Jean Lagac\'e}
\address{D\'epartement de math\'ematiques et de statistique, Universit\'e de Montr\'eal, CP 6128 succ Centre-Ville, Montr\'eal, QC H3C 3J7, Canada}
\email{lagacej@dms.umontreal.ca}

\author[I. Polterovich]{Iosif Polterovich}
\address{D\'epartement de math\'ematiques et de statistique, Universit\'e de Montr\'eal, CP 6128 succ Centre-Ville, Montr\'eal, QC H3C 3J7, Canada}
\email{iossif@dms.umontreal.ca}

\author[A. Savo]{Alessandro Savo}
\address{Dipartimento SBAI, Sezione di Matematica
Sapienza Universit\`a di Roma, Via Antonio Scarpa 16
00161 Roma, Italy}
\email{alessandro.savo@sbai.uniroma1.it}  

\keywords{Steklov problem, cuboids, spectral asymptotics, lattice counting.}

\date{}

\begin{abstract}
The paper is concerned with the Steklov eigenvalue problem on cuboids
of arbitrary dimension. We prove a two-term asymptotic formula for the
counting function of Steklov eigenvalues on cuboids in dimension
$d\ge~3$.  Apart from the standard Weyl term, we calculate explicitly
the second term in the asymptotics, capturing the contribution of the
$(d-2)$--dimensional facets of a cuboid. Our approach is  based on
lattice counting techniques. While this strategy is similar to the one
used for the Dirichlet Laplacian, the Steklov case carries additional
complications.  In particular, it is not clear how to establish
directly the completeness of the system of Steklov eigenfunctions
admitting separation of variables. We prove this result  using a
family of auxiliary Robin boundary value problems.  Moreover,  the
correspondence between the Steklov eigenvalues and lattice points is
not exact, and hence more delicate analysis is required to obtain
spectral asymptotics.  Some other related results are presented, such
as an isoperimetric inequality for the first Steklov eigenvalue, a concentration property of high frequency Steklov
eigenfunctions and applications to spectral determination of cuboids.
 \end{abstract}
\maketitle

\section{Introduction and main results}
\subsection{Asymptotics of the Steklov spectrum}
The Steklov eigenvalues of a boun\-ded Euclidean domain
$\Omega\subset\R^d$ are the real numbers $\sigma\in\R$ for which there
exists a nonzero harmonic function $u:\Omega\rightarrow\R$ such that
$\del_n u = \sigma u$ on the boundary $\del\Omega$. Here $\partial_n$
denotes the outward normal derivative, which exists
almost everywhere provided the boundary $\partial\Omega$ is
Lipschitz. Under this assumption, it is known that for $d\geq 2$ the Steklov spectrum is
discrete (see \cite{Agranovich})  and is given by the increasing sequence of  eigenvalues
$0 = \sigma_0 < \sigma_1 \le \sigma_2 \le \dotso \nearrow \infty,$
where each eigenvalue is repeated according to its multiplicity.
The counting function $N:\R\to\N$ is then defined by
$N(\sigma):=\#\{j\in\N\,:\,\sigma_j<\sigma\}.$
For domains with smooth boundary, one can show using pseudodifferential
techniques
that the counting function  satisfies Weyl's law
\begin{gather}\label{asymptotics:smooth}
  N(\sigma)= \frac{\omega_{d-1}}{(2\pi)^{d-1}}\Vol_{d-1}(\del \Omega)
  \sigma^{d-1}+O(\sigma^{d-2})\quad\mbox{ as }\sigma\nearrow+\infty,
\end{gather}
where $\omega_{d-1}$ is the measure of the unit ball
$B_1(0)\subset\R^{d-1}.$  The remainder estimate in \eqref{asymptotics:smooth} is sharp and attained  on a round ball.
Moreover, a two-term asymptotic formula for the counting function holds under a non-periodicity condition of the geodesic flow on 
$\partial \Omega$ (see \cite[formula (5.1.8)]{PS}).

Understanding  precise  asymptotics for Steklov eigenvalues on domains
with singularities, such as corners and edges, is significantly more
challenging, since   pseudodifferential techniques do not work in this
case (see \cite[Section 3]{girpoltJST} for a discussion).
Using variational  methods,  one can prove a one-term Weyl asymptotic formula  that holds for any piecewise $C^1$ Euclidean domain (see \cite{Agranovich}):
\begin{gather}
\label{C1Weyl}  
N(\sigma)= \frac{\omega_{d-1}}{(2\pi)^{d-1}}\Vol_{d-1}(\del \Omega)
  \sigma^{d-1}+o(\sigma^{d-1})\quad\mbox{ as }\sigma\nearrow+\infty.
\end{gather}
However, in order to get sharper asymptotics,  one needs to understand the contribution of  singularities to the counting function. In two dimensions, some results in this direction have been recently obtained in \cite{LPPS}. 
In the present paper we aim to explore the most basic  higher-dimensional example: the Euclidean cuboids.

\subsection{Main result}
Given $d\in\N$, the \emph{cuboid}\footnote{Cuboids are also often referred to as \emph{boxes}, \emph{$d$-orthotopes} or
\emph{hyperrectangles}. The term ``cuboids'' appears to be more common in recent literature on spectral geometry (see \cite{GL,vBG}).}
with parameters
$a_1,\dotsc,a_d>0$ is defined as a product of the intervals
$$\Omega=(-a_1,a_1)\times (-a_2,a_2)\times\dotso\times(-a_d,a_d)\subset\R^d.$$ If $ a_1=a_2=\dots=a_d$ we say that $\Omega$  is  a \emph{cube}.
The main result of this paper is the following theorem.
\begin{thm}\label{thm:main}
  Let $\Omega\subset\R^d$ be the cuboid with parameters
  $a_1,\dotsc,a_d>0$.   For $d\geq 3$, the counting function of Steklov eigenvalues satisfies a two-term asymptotic formula as $\sigma \to \infty$:
  \begin{equation}
\label{counting:higher}
 N(\sigma) = C_1 \Vol_{d-1}(\del \Omega) \sigma^{d-1} + C_2 \Vol_{d-2}
 (\del^2\Omega) \sigma^{d-2}  + \bigo{\sigma^{\eta}},
\end{equation}
where $\del^2\Omega$ denotes the union of all the $(d-2)$-dimensional facets of $\Omega$.
Here  $\eta=2/3$ for $d=3$ and $\eta=d-2-\frac{1}{d-1}$ for $d\geq 4$. The
constants $C_1$ and $C_2$ are given by
 \begin{align*}
  C_1 = \frac{\omega_{d-1}}{(2\pi)^{d-1}}
  \end{align*}
  and
  \begin{align*}
  C_2 = \frac{2^{\frac{d-2}2} \omega_{d-2}}{(2\pi)^{d-2}} -  \frac{2G_{d-1,1}}{\pi^{d-1}} -  \frac{\omega_{d-2}}{2(2\pi)^{d-2}}
 \end{align*}

where
\begin{equation*}
 G_{d-1,1} = \underbrace{\int_0^{\pi/2} \dotso \int_0^{\pi/2}}_{d-2}
 \arccot\left(\prod_{j=1}^{d-2} \csc \theta_j \right)
 \prod_{k=1}^{d-2} \sin^k(\theta_k) \de \theta_1 \ldots \de \theta_{d-2}.
\end{equation*}
For $d=2$, the counting function  admits a one-term asymptotics
  \begin{equation*}
    N(\sigma) = \pi^{-1}\Vol_{1}(\del \Omega) \sigma + \bigo 1.
  \end{equation*}
\end{thm}

\begin{rem}
 \label{weylconst}
It can be shown that $C_2>0$ for all $d \ge 3$, see  Appendix~\ref{appendix:positive}. 
  The constants $G_{d,1}$ are special cases of constants $G_{p,q}$
  which will be introduced in Section \ref{section:asymptotics}.  
The constants $G_{2,1}$ and $G_{3,1}$ can be computed explicitly as
\begin{equation*}
 \begin{aligned}
  G_{2,1} &= \frac{1}{2}\left(-1 +\sqrt 2\right) \pi\\
  G_{3,1} &= \frac{1}{8}\left(- 2 + \pi\right) \pi.
 \end{aligned}
\end{equation*}
\end{rem}
\begin{rem}
  For $d=2$, the above asymptotics also follows from \cite[Corollary 1.6.1]{LPPS}.
\end{rem}

\begin{rem}
    For $d=3$, Theorem \ref{thm:main} predicts that
    $$\mathcal{R}(\sigma)=\frac{N(\sigma) - C_1 \Vol_{2}(\del \Omega) \sigma^{2} - C_2
      \Vol_{1}(\del^2\Omega) \sigma}{\sigma^{2/3}}$$
    is
    a bounded function of $\sigma$.
    In order to validate the expression for the constant $C_2$ obtained in Theorem \ref{thm:main}, we have checked numerically that this claim holds, using the approximate eigenvalues introduced in Section \ref{section:asymptotics} on a cube with side lengths $2$. Figure~\ref{fig:numerics} shows that $|\mathcal{R}(\sigma)|\le 3$ for $\sigma < 750$  which corresponds to approximately  a million eigenvalues. 
    \begin{figure}
      \centering
      \includegraphics[width=10cm]{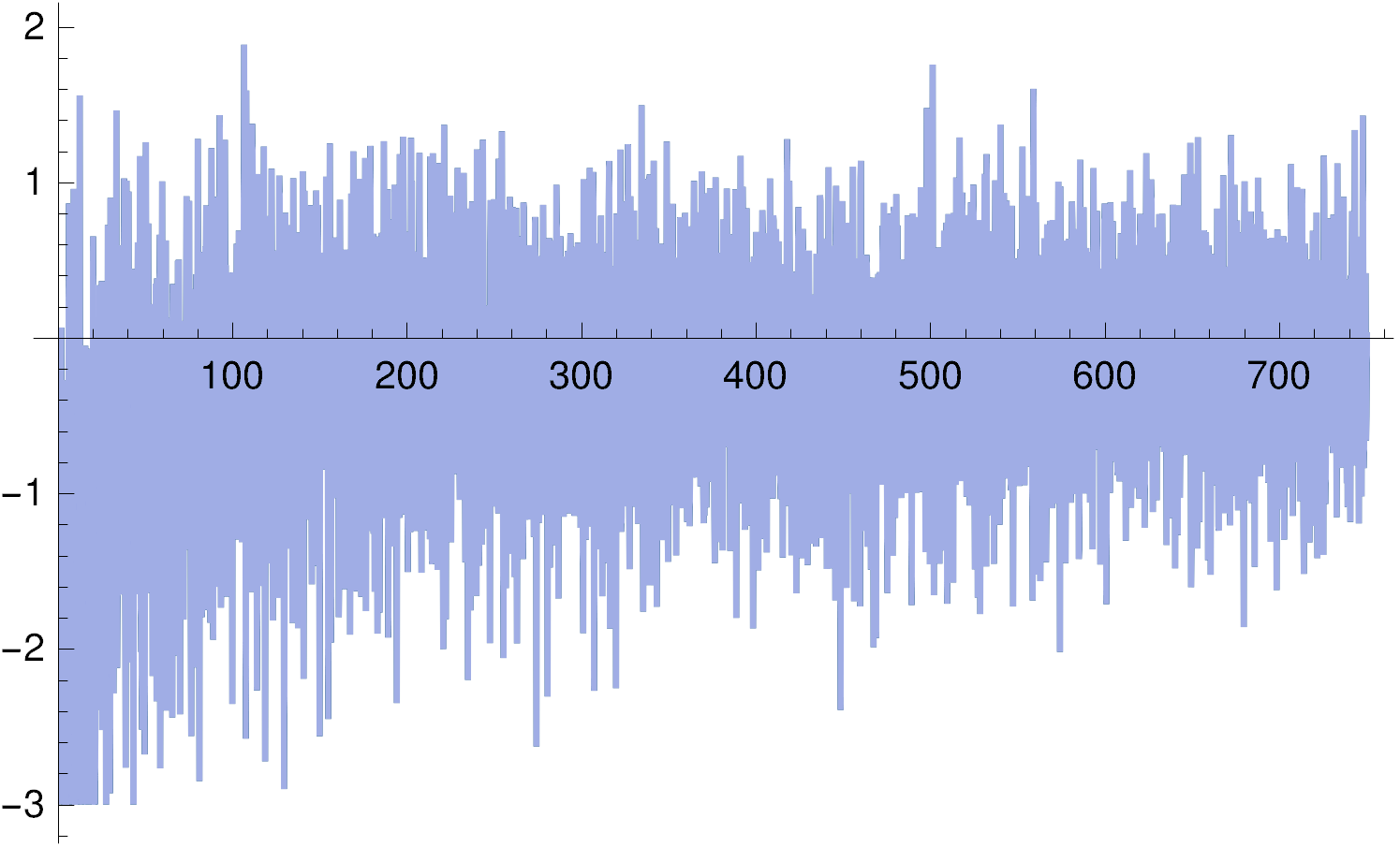}
      \caption{$\frac{N(\sigma) - C_1 \Vol_{2}(\del \Omega) \sigma^{2} - C_2
      \Vol_{1}(\del^2\Omega) \sigma}{\sigma^{2/3}}$ for $\sigma < 750$.}
      \label{fig:numerics}
    \end{figure}
\end{rem}

\subsection{Outline of the proof}
The proof of Theorem~\ref{thm:main} is given in  Section
\ref{section:asymptotics}. The outline of the argument is as follows.
First, we show that the Steklov eigenvalue problem on a cuboid admits
separation of variables, see  Lemma~\ref{Lemma:separated} below.
Separation of variables yields eigenfunctions that are produts of
trigonometric, hyperbolic and possibly linear factors. One can check
that the number of eigenvalues corresponding to eigenfunctions
containing linear terms is at most finite, see Theorem
\ref{thm:eigenlattice}. The same theorem also shows that the
eigenvalue counting problem can be reduced to a family of  approximate
lattice counting problems. More specifically, given  $1\le p\le d$, we
consider the counting function $N_p$ of eigenvalues corresponding to
eigenfunctions with exactly $p$ trigonometric factors. It turns out
that for each $p>1$, the counting function $N_p$ satisfies a  two-term
asymptotic formula, see Proposition~\ref{Npsigma}. The functions $N_p$
for  $p=d-1$ and $p=d-2$ are the dominant ones. In particular,
the main term in \eqref{counting:higher} corresponds to the main term in
the asymptotics for $N_{d-1}$. The second term in
\eqref{counting:higher} is obtained as as a sum of the main  term in
the asymptotics of $N_{d-2}$ and the second term in $N_{d-1}$. The
latter also splits into two parts: one is the standard contribution of
overcounted lattice points (see Lemma~\ref{lem:overcounted}), and the
other has to do with the geometry of the domain $E_\sigma$
defined by \eqref{eq:Esigma} arising in the lattice counting problem.
While this domain $E_\sigma$ converges to a ball as $\sigma \to
\infty$, the approximation produces an error that contributes to the
second term of \eqref{counting:higher}. This explains why the
coefficient $C_2$ is represented by a sum of three constants. Note
that while two of these 
constants are negative, the coefficient $C_2$ is always positive, see
Appendix~\ref{appendix:positive}. 

\subsection{Discussion}
The second term in Weyl asymptotics \eqref{counting:higher} for cuboids  could be compared with the corresponding term in the asymptotic expression \cite[formula (5.1.8)]{PS} mentioned earlier, which holds on  smooth manifolds with boundary,  satisfying a non-periodicity condition. Recall that in the smooth case, the second term is proportional to the integral of the mean curvature of the boundary. A similar interpretation could be given to the second term in \eqref{counting:higher}, if an analogue of the mean curvature for cuboids is   thought of as  a $\delta$-function supported on the union of the $(d-2)$-dimensional facets.

It would be very interesting to establish an analogue of Theorem \ref{thm:main} for arbitrary Euclidean polyhedra and, more generally, for Riemannian manifolds with edges, satisfying certain non-periodicity  assumptions. While the present paper was in the final stages of preparation, V. Ivrii \cite{Iv} informed us on his work in progress in this direction.   
We believe that a two-term Weyl asymptotic formula \eqref{counting:higher} holds for any polyhedron in dimension $d\ge 3$, with the coefficients $C_1$ and $C_2$ depending on the dimension and the angles between the $(d-1)$-dimensional facets of a polyhedron. 

Another promising direction of further research in the subject is to explore the asymptotic expansion for the Steklov heat trace on Euclidean polyhedra,  as well as on arbitrary Riemannian manifolds with edges. In particular, one could  ask whether the Steklov spectral asymptotics contains  information on the lower-dimensional facets of  polyhedra. While the Weyl asymptotics does not appear to be accurate enough for that purpose, the Steklov heat trace asymptotics  is likely to  give  a positive  answer to this question.  We intend to explore it elsewhere.
 
\begin{rem}
The existence of a  two-term asymptotic formula for the counting function of Steklov eigenvalues on a cube was claimed  earlier in \cite{pinascorossi}. However, the proof of this claim  contained a miscalculation invalidating the argument. Indeed, in the beginning of \cite[Section 3]{pinascorossi}, the authors write down the boundary condition at $x_i=0$ in  case $\beta_i<0$ and get $c_1\sqrt{|\beta_i|}=\lambda c_2$, while it should be $-c_1\sqrt{|\beta_i|}=\lambda c_2$, since the normal derivative at $x_i=0$ is $-\partial_i$. Due to this missing minus sign, the authors obtain the equation $\sin(\sqrt{\beta_i})=0$ leading to an exact correspondence between Steklov eigenvalues and lattice points. However,  in reality this correspondence is only approximate (see subsection \ref{reduc}), and therefore counting eigenvalues is a significantly more difficult task. Note also that the completeness of eigenfunctions admitting separation of variables was not justified  in~\cite{pinascorossi}.
\end{rem}

\subsection{An isoperimetric  inequality for the first Steklov eigenvalue} 
Given a cuboid $\Omega\subset\R^d$ with parameters $a_1,\dotsc,a_d>0$,
let $\Omega^\star$ and $\Omega^\sharp$ be the cubes such that
$$\Vol_{d-1}\partial\Omega^\star=\Vol_{d-1}\partial\Omega
\qquad\mbox{ and }\qquad\Vol_{d}\Omega^\sharp=\Vol_{d}\Omega.
$$
\begin{thm}\label{thm:isoperimetric}
  For any cuboid $\Omega$,
  \begin{itemize}
  \item $\sigma_1(\Omega^\star)\geq\sigma_1(\Omega)$, with equality if and
  only if $\Omega^\star=\Omega$;
  \item $\sigma_1(\Omega^\sharp)\geq\sigma_1(\Omega)$, with equality if and
  only if $\Omega^\sharp=\Omega$.
  \end{itemize}
\end{thm}

The proof of the theorem is presented  in Section \ref{proof:isoperim}. In a way, it is not surprising that the cube, being the most symmetric of all cuboids, maximizes $\sigma_1$ under both  volume and surface  area restrictions.
Theorem \ref{thm:isoperimetric} could  be compared with the well-known Weinstock's inequality \cite{wein}
stating that the disk is a unique maximizer for
$\sigma_1$ among planar simply connected
domains with a given perimeter (see also a recent generalization of this result for convex domains in higher dimensions obtained in \cite{BFNT}), as well as with Brock's result
\cite{brock} which states that balls are unique maximizers among 
Euclidean domains $\Omega\subset\R^d$ with prescribed $d$--volume.

It follows from Theorem \ref{thm:isoperimetric} that any cube is spectrally determined among all cuboids.
\begin{cor}
  Let $\Omega\subset\R^d$ be a cuboid which is isospectral to
  the cube $\Omega_a\subset\R^m$ with side lengths $2a>0$. Then
  $d=m$ and $\Omega=\Omega_a$. 
\end{cor}
\begin{proof}
  It follows from Theorem \ref{thm:main} that $d=m$ and
  $\Vol_{d-1}(\partial\Omega)=\Vol_{d-1}(\partial\Omega_a)$. Moreover,
  since $\sigma_1(\Omega)=\sigma_1(\Omega_a)$, the
  conclusion follows from the uniqueness of the maximizer in Theorem \ref{thm:isoperimetric}.
\end{proof}

Note that a similar corollary with an almost identical proof holds for planar simply-connected domains, among which the disk is spectrally determined, using the case of equality in Weinstock's theorem \cite{wein}.

Is still unknown whether there exist nonisometric Steklov isospectral Euclidean domains. Our results imply that if two rectangles are Steklov isospectral, they are isometric. 
\begin{cor}
\label{cor:rect}
 The Steklov spectrum of a rectangle uniquely determines its side lengths.
\end{cor}

The proof of this corollary is  presented in 
  Section~\ref{section:spectraldetermination}. Let us conclude the introduction with the following conjecture:
\begin{conj}
Any two Steklov isospectral cuboids are isometric.
\end{conj}

\subsection*{Plan of the paper}
In Section 2,  we explore the structure of Steklov eigenvalues and eigenfunctions on cuboids. In particular, in subsection \ref{sepvar} we describe separation of variables and prove that it yields a complete system of Steklov eigenfunctions.  In subsection \ref{classeig} a classification of eigenfunctions is presented based on the number of linear, trigonometric and hyperbolic terms, which is later used in subsection \ref{reduc} to reduce the problem of counting eigenvalues to counting approximate lattice points. Theorem
\ref{thm:main} is proved in Section \ref{section:asymptotics}. This is the most technicallly involved part of the paper, involving tools from analytic number theory and Fourier analysis. 
Other results of the paper are proved in Section \ref{other}. In particular, a somewhat surprising observation that Steklov eigenfunctions may concentrate on lower dimensional facets of cuboids is presented in subsection \ref{concent}. Subsections \ref{proof:isoperim} and \ref{section:spectraldetermination} provide the proofs of Theorem \ref{thm:isoperimetric} and 
Corollary \ref{cor:rect}. Appendix  \ref{appendix:arccot}  contains  the  proof  of an auxiliary Lemma \ref{lem:arccot} used in subsection \ref{many}.  In Appendix  \ref{appendix:positive}  we 
justify the positivity of the constant $C_2$ as stated  in Remark \ref{weylconst}.

\begin{rem}
  Right before submitting our paper on the archive, we learned of the
  preprint \cite{ArnoldTan} which discusses Steklov eigenvalues of
  rectangles and cuboids of dimension 3. Note that \cite[Conjecture
    3.1]{ArnoldTan} immediately follows from our Proposition~\ref{prop:firsteigen}.
\end{rem}

\subsection*{Acknowledgments}
We would like to thank Leonid Parnovski for useful discussions. In particular, they lead us to the results presented in subsection \ref{concent}. Part
of this work was accomplished in 2016 when A.G. visited the Institut de math\'ematiques
de Neuch\^atel, and  its hospitality is gratefully acknowledged.
Research of A.G. is supported by NSERC and FRQNT.
Research of J.L. is supported by the Alexander Graham Bell Canada Graduate Scholarship.
Research of I.P. is supported by NSERC, FRQNT and Canada Research Chairs Program.
Research of A.S. is supported by PRIN 2015.
\section{Eigenfunctions and separation of variables}\label{section:eigenfunctions}
\subsection{Separation of variables}
\label{sepvar}
The following lemma shows that  the method of separation of variables is
applicable to the computation of the Steklov spectrum of a product of
compact manifolds with boundary. In particular, we justify completeness of the system of 
Steklov eigenfunctions admitting separation of variables.
\begin{lem}\label{Lemma:separated}
  Let $M_1$ and $M_2$ be smooth compact Riemannian manifolds with 
  boundary. Let $\sigma\geq 0$ be a Steklov
  eigenvalue of the product manifold $M=M_1\times M_2$ with the eigenspace
  $F_\sigma \subset L^2(M)$. There exists a basis 
  $(u^{(1)},\dotsc,u^{(m)})$ of $F_\sigma$ such that each
  $u^{(j)}:M_1\times M_2\rightarrow\R$ is separable:
 \begin{equation*}
  u^{(j)}(x_1,x_2) = u^{(j)}_1(x_1) u^{(j)}_2(x_2),\qquad 1\leq j\leq m,
 \end{equation*}
 where $u^{(j)}_1:M_1\rightarrow\R$ and $u^{(j)}_2:M_2\rightarrow\R$.
\end{lem}
\begin{proof}
 Consider the Robin problem with parameter $\sigma \ge 0$ on $M$
 \begin{equation*}
  \begin{cases}
   \Delta u + \lambda u = 0 & \text{in } M,\\
   \del_n u = \sigma u & \text{on } \del M.
  \end{cases}
 \end{equation*}

It is well known that the Robin problem on $M$
 admits separation of variables, since $L^2(M) = L^2(M_1) \otimes
 L^2(M_2)$ is a product space, see \emph{e.g.} \cite[Section 11.5]{Strauss}. 
 The number $\sigma\ge0$ is a Steklov eigenvalue of $M$ if and only if
 $0$ is an eigenvalue of the Robin problem with parameter
 $\sigma$, and the corresponding eigenspace is the same for both problems.
 Since one can find a separated eigenbasis
 for $F_\sigma$ by virtue of it being a Robin eigenspace on $M$, it then suffices to use
 the same basis for $F_\sigma$ when we consider it as a Steklov
 eigenspace.
\end{proof}
\begin{rem}

It is not easy to show  directly that the traces of all separable Steklov eigenfunctions form a basis in $L^2(\partial M)$,  since the boundary $\partial M$ of a product manifold is not itself a product manifold.
\end{rem}
\begin{rem}
  Lemma \ref{Lemma:separated} yields completeness of the system of
  separable Steklov eigenfunctions on cuboids. Surprisingly, a complete
  proof of this result has not appeared in the literature even in the
  case of rectangles. Note that the completeness argument for the square
  presented in \cite[Section 3]{girpoltJST} does not extend to arbitrary
  rectangles, contrary to the claim made in \cite[Section
    4]{AuchmutyCho} and in \cite{ArnoldTan}. Indeed, the proof given in
  \cite{girpoltJST} uses in a crucial way the diagonal symmetries of the
  square, which allow to use a connection to the vibrating beam problem
  via mixed Steklov-Neumann-Dirichlet problems on an isosceles right
  triangle.
\end{rem}

Let $d\in\N$ and consider the cuboid $\Omega$ with parameters
$a_1,\dotsc,a_d>0$.
Because $\Omega$ is a product of compact intervals, it follows from
Lemma \ref{Lemma:separated} that there exists a complete set
$\set{u_j}_{j\in \N_0}$ of separated Steklov eigenfunctions on $\Omega$. 
Consider a function $u:\Omega\rightarrow\R$
given by the product $u(x)=u_1(x_1)\dots u_d(x_d)$, 
where $u_j:~[-a_j,a_j]\rightarrow\R$. Requiring $u$ to be a Steklov
eigenfunction with eigenvalue $\sigma\geq 0$ leads to numbers
$\lambda_1,\lambda_2,\dotsc,\lambda_d\in\R$ such that
\begin{equation}\label{eq:steklovsepare}
\begin{cases}
 u_j'' + \lambda_ju_j=0\quad\mbox{ on }\quad (-a_j,a_j), \\
 u_j'(a_j) = \sigma u_j(a_j), \\
 -u_j'(-a_j) = \sigma u_j(-a_j),
\end{cases}
\end{equation}
subject to the harmonicity condition
\begin{equation}\label{eq:conditiondesomme}
 \sum_{j=1}^d \lambda_j = 0.
\end{equation}
The following lemma describes the eigenvalues and eigenfunctions of the auxilary
one-dimensional Steklov spectral problem \eqref{eq:steklovsepare} with a
parameter $\lambda\in\R$.
\begin{lem}\label{Lemma:oneD}
  Let $\lambda\in\mathbb{R}$. The non-zero solutions $\phi:[-a,a]\rightarrow\R$ of the differential equation $\phi''+\lambda\phi=0$ subject to the
  boundary conditions 
  \begin{gather*}
    \phi'(a)=\sigma\phi(a) \quad\mbox{ and }\quad
    -\phi'(-a)=\sigma\phi(-a)
  \end{gather*}
  for some constant $\sigma\geq 0$, are constant multiples of one the
  following functions:
  
  \begin{enumerate}[(i)]
  \item For $\lambda=0$, $\phi(t)\equiv 1$ and $\sigma=0$ or
    $\phi(t)=t\mbox{ and }\sigma=a^{-1}.$
    
\smallskip
\item For $\lambda=\alpha^2>0$, one of
    \begin{align*}
      \phi(t)&=\sin(\alpha t)\hspace{.5cm}\mbox{ with }\sigma = \alpha\cot(\alpha a),\\
      \phi(t)&=\cos(\alpha t)\hspace{.5cm}\mbox{ with }\sigma = -\alpha\tan(\alpha a).
    \end{align*}
    In other words, for each $\ell \in \set{0,1}$,
    $\sigma = \alpha \cot\left(\alpha a + \ell \frac \pi 2\right)$ is an
    eigenvalue.
  
\smallskip

  \item For $\lambda=-\beta^2<0$, one of
    \begin{align*}
      \phi(t)&=\sinh(\beta t)\hspace{.5cm}\mbox{ with }\sigma = \beta\coth(\beta a)\\
      \phi(t)&=\cosh(\beta t)\hspace{.5cm}\mbox{ with }\sigma = \beta\tanh(\beta a).
    \end{align*}
    In other words, for each $j \in \set{-1,1}$,
    $\sigma= \beta \tanh(\beta a)^{j}$ is an eigenvalue.
  \end{enumerate}
\end{lem}

It will be useful to introduce a uniform notation for these
eigenvalues. Given $a>0$ and $\ell\in\{0,1\}$, let
\begin{gather*}
  T_{a,\ell}(x)=x\cot\left(ax + \ell \frac \pi 2\right)=
  \begin{cases}
    x\cot(ax)&\mbox{for }\ell=0,\\
    -x\tan(ax)&\mbox{for }\ell=1,
  \end{cases}
\end{gather*}
 and 
  \begin{gather*}
  H_{a,\ell}(x)
  =
  \begin{cases}
    x\coth(ax)&\mbox{for }\ell=0,\\
    x\tanh(ax)&\mbox{for }\ell=1.
  \end{cases}
\end{gather*}

It follows from Lemma \ref{Lemma:oneD} that separable eigenfunctions
are products of linear factors, trigonometric factors (the function
$\sin$ for $\ell=0$, and $\cos$ for
$\ell=1$) and hyperbolic factors (the function $\sinh$ for $\ell=0$,
and $\cosh$ for
$\ell=1$). A
careful accounting of these will be presented.

\subsection{Classification of eigenfunctions}
\label{classeig}
It follows from the previous paragraph that there is a complete set of
Steklov eigenfunctions which are given by products of linear, trigonometric
and hyperbolic factors. They are of the form
\begin{gather}\label{function:eigenseparated}
  u(x_1,\dots,x_d)=\prod_{i\in\tau_0}x_i\prod_{j\in\tau_1}\mbox{Trig}_j(\alpha_jx_j)\prod_{k\in\tau_2}\mbox{Hyp}_k(\beta_kx_k)
\end{gather}
where $\tau_0,\tau_1,\tau_2$ are disjoint subsets of
$S_d:=\{1,2,\dots,d\}$ such that
$\tau_0\cup\tau_1\cup\tau_2=S_d$, and each
$\mbox{Trig}_j\in\{\sin,\cos\}$ and $\mbox{Hyp}_k\in\{\sinh,\cosh\}$. In
order for this function to be a Steklov eigenfunction corresponding to
the eigenvalue $\sigma>0$, the function $u$ must be harmonic. This
amounts to the following restatement of condition
  \eqref{eq:conditiondesomme} in terms of the 
constants $\alpha_j$ and $\beta_k$:
\begin{gather}\label{equation:harmonic}
  \sum_{j \in \tau_1} \alpha_j^2 = \sum_{k \in \tau_2} \beta_k^2.
\end{gather}
This equation will be called the \emph{harmonicity condition}. Moreover, the spectral parameter $\sigma$ has to be the same on each
face of the cuboid. By Lemma \ref{Lemma:oneD} this translates into the
following equations, called the \emph{compatibility conditions}: 
\begin{gather}\label{equation:compatibility}
  \sigma=
  \begin{cases}
    a_i^{-1} &\quad\mbox{for }i\in\tau_0,\\
    T_{a_i,\ell(i)}(\alpha_i)&\quad\mbox{for }i\in\tau_1,\\
    H_{a_i,\ell(i)}(\beta_i)&\quad\mbox{for }i\in\tau_2.
  \end{cases}
\end{gather}
Here the function $\ell:S_d\rightarrow\{0,1\}$ is used
to specify which trigonometric and hyperbolic functions are used,
according to the convention introduced in Lemma \ref{Lemma:oneD}.
The corresponding eigenfunction \eqref{function:eigenseparated} is
then given precisely by the product of the factors $u_i:[-a_i,a_i]\rightarrow\R$
which are specified by
\begin{gather}\label{function:separatedfactors}
  u_i(x_i)=
  \begin{cases}
    \Trig_{\ell(i)}(\alpha_i x_i)&\mbox{ for }i\in\tau_1,\\
    \Hyp_{\ell(i)}(\beta_ix_i)&\mbox{ for }i\in\tau_2,\\
    x_i&\mbox{ otherwise},
  \end{cases}
\end{gather}
where $\Trig_0=\sin$, $\Trig_1=\cos$, $\Hyp_0=\sinh$ and $\Hyp_1=\cosh$.

Note that any separated eigenfunction that
has a linear factor $u_j(x_j)=x_j$ contributes the eigenvalue
$\sigma=a_j^{-1}$ to the spectrum. Since  the multiplicity of each
eigenvalue is finite, this can occur at most a finite number of
times. We summarize the above mentioned facts in the following theorem.
\begin{thm}\label{thm:separation}

Let $p\in\{1,\dotsc,d-1\}$, and  let $\CT_p$ be the
    set of all ordered bipartitions 
    $\tau\krn=\krn(\tau_1,\tau_2)$ of $\set{1,\dotsc,d}$ in the sets of cardinality $p$
    and $q = d-p$. For each $\tau\in\CT_p$ and any
    $\ell:\tau_1\cup\tau_2\rightarrow \{0,1\}$, let
    $S_{\tau,\ell}$ be the  set of all numbers $\sigma>0$ for which
    there exist positive numbers $\alpha_i$ for $i\in\tau_1$ and
    $\beta_j$, for $j\in\tau_2$, which solve
    \begin{equation*}
      \sigma=T_{a_i,\ell(i)}(\alpha_i)=H_{a_j,\ell(j)}(\beta_j) \qquad\forall i \in \tau_1, j \in \tau_2
    \end{equation*}
    subject to the constraint
    \begin{equation*}
      \sum_{i \in \tau_1} \alpha_i^2 = \sum_{j \in \tau_2} \beta_j^2.
    \end{equation*}
Denote also by $S_{0}$ the collection of Steklov eigenvalues corresponding to separated eigenfunctions having a linear factor.
Then the Steklov spectrum of a cuboid $\Omega$ is given by the union of $S_0$ which contains at most finitely many elements,  
and the families $S_{\tau,l}$ for all possible choices of $\tau$ and $\ell$.
\end{thm}

\subsection{Reduction to approximate lattice counting}
\label{reduc}
We will now give a  more precise description of the
spectrum by constructing a correspondence between the Steklov eigenvalues of cuboids and the vertices of certain lattices. 

Let $\Omega$ be a cuboid with parameters $a_1,\dotsc,a_d$. Let 
$p\in\{1,\dotsc,d-1\}$ represent the number of trigonometric
factors of a separated eigenfunction without linear factors. Each bipartition
$\tau=(\tau_1,\tau_2)\in \CT_p$ then corresponds  to a separated
eigenfunction of the form
\begin{gather}\label{eigenfunction:formbipart}
  u(x_1,\dots,x_d)=\prod_{j\in\tau_1}\mbox{Trig}_j(\alpha_jx_j)\prod_{k\in\tau_2}\mbox{Hyp}_k(\beta_kx_k)
\end{gather}
 Let  $\N_0=\{0,1,2,\dotsc\}$ be the set of nonnegtive integers.
Given $\bn\in\N_0^p$,  let
$$I_\bn=I_{\bn,p,\tau}:=
\prod_{i\in\tau_1}
\left( \frac{n_i \pi}{2 a_i},  \frac{(n_i+1) \pi}{2 a_i}\right]\subset\R^p.$$

The boxes $I_\bn$ are fundamental domains of a lattice. The following theorem shows that each box gives rise to a cluster of at most $2^q$ eigenvalues and, moreover, the boxes
$I_\bn$ with $\bn \in \N^p$ and $|\bn|$ large enough correspond to precisely $2^q$ eigenvalues.
\begin{thm}\label{thm:eigenlattice}
Given $p\in \{1,\dotsc,d-1\}$, and $q = d-p$, let $\tau \in \CT_p$ specify the
position of trigonometric and hyperbolic factors of eigenfunctions
of the form \eqref{eigenfunction:formbipart}. The following assertions hold:

\smallskip

\noindent {\rm (i)} Eigenfunctions of the form \eqref{eigenfunction:formbipart} form a complete system of Steklov eigenfunctions on a cuboid up to a finite number of eigenfunctions containing linear
factors.

\smallskip

\noindent {\rm (ii)} For each $\bn \in \N^p$, there exist at most $2^q$ eigenfunctions of
the form \eqref{eigenfunction:formbipart} with $\balpha \in I_\bn$.

\smallskip
\noindent {\rm (iii)} There exists a number $N \in\N$, such that for every $\bn \in \N^p$ with $|\bn| > N $, there are \emph{exactly} $2^q$ eigenfunctions of the form \eqref{eigenfunction:formbipart} with
$\balpha\in I_\bn$. The corresponding eigenvalues
$\sigma_\bn^{(k)}$, with $k \in \set{1,\dotsc,2^q}$,
satisfy
\begin{equation} \label{eq:sigmabnk}
\sigma_\bn^{(k)} = \frac{|\balpha_\bn|}{\sqrt q} + \bigo{|\bn|^{-\infty}}
\end{equation}
for some $\balpha_\bn \in I_\bn$.

\smallskip

\noindent {\rm (iv)} There exist only finitely many eigenfunctions of the form \eqref{eigenfunction:formbipart} such that $\bn \in \N_0^p \setminus \N^p$.
For each $\bn \in \N_0^p \setminus \N^p$, there are at most $2^q$ eigenfunctions of the form \eqref{eigenfunction:formbipart} with $\balpha \in I_\bn$.
\end{thm}
Assertions (ii) and (iii) essentially say that up to a finite number of boxes, there is always exactly $2^q$ solutions in the box $I_\bn$, while assertion (iv) says that while some boxes touching the coordinate hyperplanes $\set{x_j = 0}$ might contain solutions, this will only happen a finite number of times. This means that while all the three cases are needed to fully describe the spectrum, asymptotically we can only count eigenvalues described by (iii), up to a $\bigo 1$ error.

\begin{proof}[Proof of Theorem \ref{thm:eigenlattice}]

Assertion (i) is a direct consequence of  Lemmas \ref{Lemma:separated} and \ref{Lemma:oneD}. 
In order to prove assertion (ii), for each $\ell : S_d \to \set{0,1}$ and $\bn \in
  \N^p$ we will show that there exists  at most one
  eigenfunction. Up to a small error, the corresponding eigenvalue will be equal to the norm
    of a point which is located in the box $I_{2\bn+\bm,p,\tau}$,
    where $\bm\in\{0,1\}^p$ is determined by the restriction of $\ell$
    to $\tau_1$. Together with the choice of $\ell$ on $\tau_2$,
    this will account for clusters of at most $2^q$ eigenvalues corresponding
    to each of the boxes $I_{\bn}$.

  \subsubsection*{Construction of  an eigenfunction.}
  For each $i\in\tau_2$, the function
  $\beta_i\mapsto H_{a_i,\ell(i)}(\beta_i)$, is increasing and
  positive for $\beta_i > 0$. It satisfies 
  $H_{a_i,\ell(i)}(\beta_i)=\beta_i+O(\beta_i^{-\infty})$ as $\beta_i\to\infty$
  and
  $$\lim_{\beta_i\to 0}H_{a_i,\ell(i)}(\beta_i)=
  \begin{cases}
    \frac{1}{a_i}  &\mbox{ if }\ell(i)=0,\\
    0&\mbox{ if }\ell(i)=1.
  \end{cases}$$
  This implies that the equations
  \begin{gather}\label{curve:hyp}
    H_{a_i,\ell(i)}(\beta_i)=H_{a_j,\ell(j)}(\beta_j) \qquad\forall i,j \in \tau_2
  \end{gather}
  define a connected curve $C_H=C_{H,p,\tau}\subset\R^q$ (the index $H$ stands for
  ``hyperbolic'') which behaves like the
  diagonal
  $$\{\bbeta\in\R^q\,:\,\beta_i=\beta_j \mbox{ for each } i,j\in\tau_2\}$$
  to infinite order as $|\bbeta|\to\infty$.
  The common value given by
  equation~(\ref{curve:hyp}) increases monotonically from some $c\geq 0$
  to infinity along the curve $C_H$ as it moves away from the origin.
  In fact, this  non-negative constant is
  $$c_\ell=\max\{0,a_i^{-1}\,:i\in\tau_2, \ell(i)=0\}.$$
  On the other hand, for each
  $i\in\tau_1$ the restricted function  
 \begin{equation} \label{eq:restriction} 
 T_{a_i, \ell(i)} :\left( \frac{n_i \pi}{ a_i}+\frac{\ell(i) \pi}{2 a_i}, \frac{n_i \pi}{ a_1} + \frac{(\ell(i) + 1)\pi}{2a_i} \right] \longrightarrow [0,\infty),
 \end{equation}
  is decreasing and surjective. Hence, for each point
  ${\bbeta}\in C_H\subset\R^q$, there exist unique numbers
  $$\alpha_i(\bbeta)\in
  \left( \frac{n_i \pi}{ a_i}+\frac{\ell(i) \pi}{2 a_i}, \frac{n_i
      \pi}{ a_1} + \frac{(\ell(i) + 1)\pi}{2a_i} \right]\qquad (\mbox{ for each }i\in\tau_1)$$ 
  such
  that
  \begin{equation}\label{eq:curve}
   T_{a_i,\ell(i)}(\alpha_i)=H_{a_j,\ell(j)}(\beta_j) \qquad\forall i \in \tau_1, j \in \tau_2.
  \end{equation}
  This defines an image curve $C_T\subset\R^p$ given by
  $$C_T=\{\alpha_i(\bbeta)\,:\,i\in\tau_1, \bbeta\in C_H\}.$$
  In other words, we have defined a continuous map
  $\balpha:C_H\longrightarrow C_T$ between these two curves. 
  It follows from (\ref{eq:restriction}) that the curve $C_T$ is
  contained in the box $I_{2\bn+\bm}$,
  where $\bm\in\{0,1\}^p$ is determined by the restriction of $\ell$
  to $\tau_1$. In particular, as the value of $|\bbeta|$ increases from
  its minimal value to $+\infty$ along the curve $C_H$, the value of $|\balpha(\bbeta)|$ is
  contained in the compact interval
  \[
   \left[\inf_{\bx \in I_{2\bn + \bm}}|\bx|, \sup_{\bx \in I_{2\bn + \bm}}|\bx| \right] \subset (0,\infty).
  \]
  Hence, if $\inf_{\bx \in I_{2\bn+\bm}}|\bx| > c_\ell$ there will be a point
  $\bbeta\in C_H$ such that $\balpha=\alpha(\bbeta)$ satisfy
  $|\balpha|\krn=\krn|\bbeta|$.  This amounts to saying that any of the common
  values given by \eqref{eq:curve} is a Steklov eigenvalue of the
  cuboid. It follows from monotonicity of each factors in Equation
  (\ref{eq:curve}) that this solution $(\balpha,\bbeta)$ is unique.
  \begin{figure}
     \includegraphics[width=10cm]{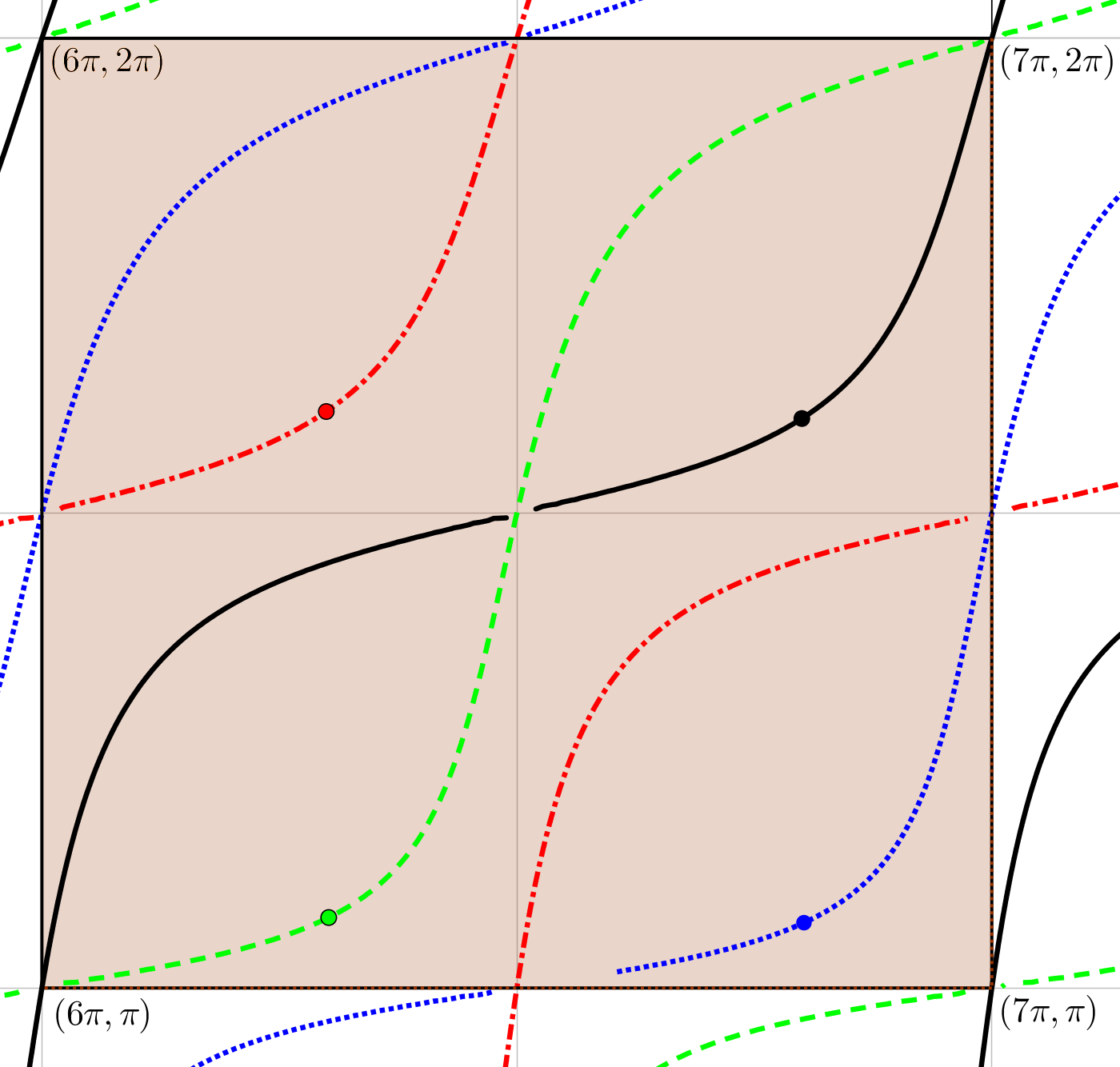}
     \caption{Various $C_T$ curves in the situation where $d=3$, $p=2$ and $\tau_1=\{1,2\}$.}
     \label{figure:curveCT}
   \end{figure}
   \begin{figure}
     \includegraphics[width=10cm]{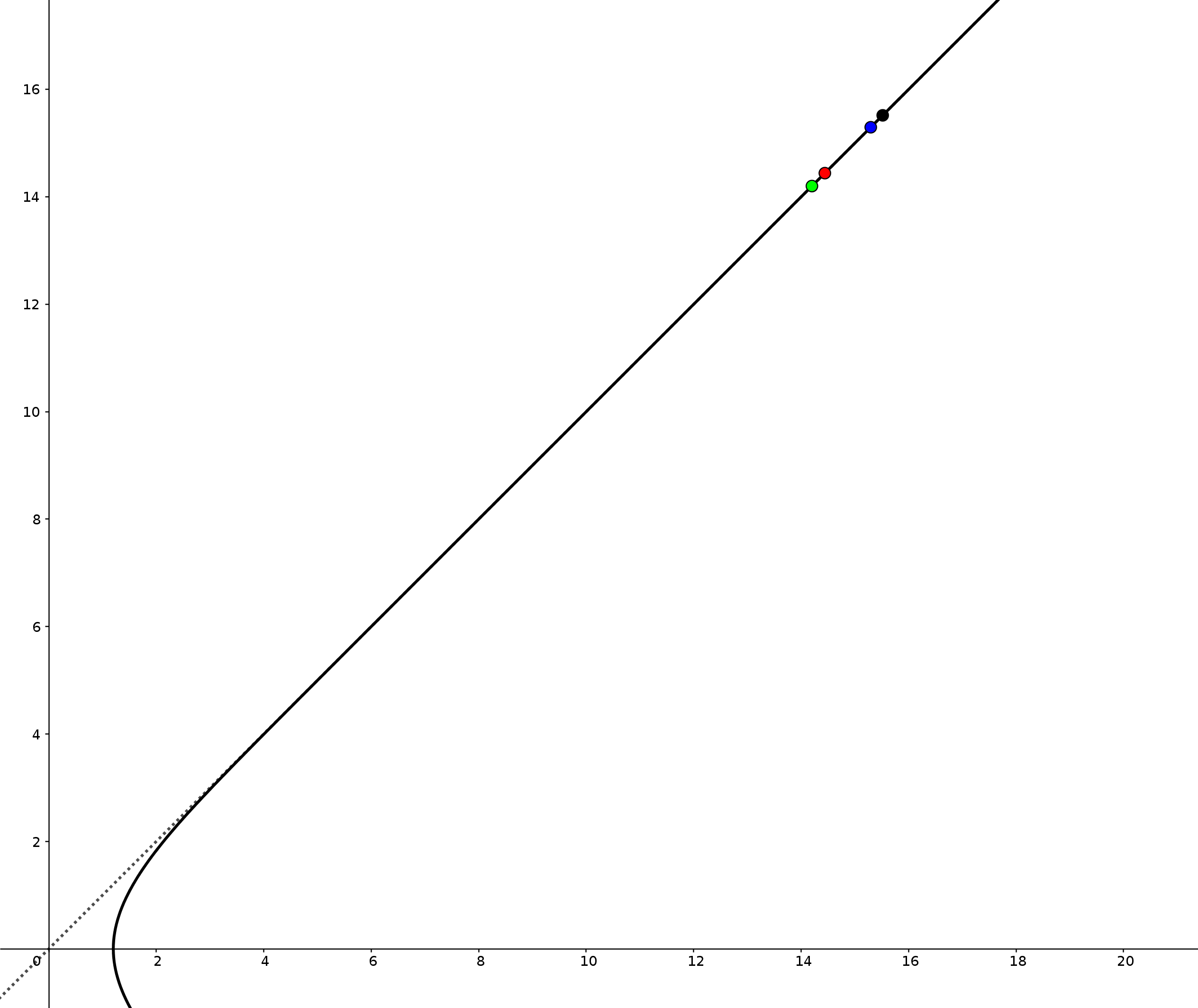}
     \caption{The curve $C_H$ corresponding to $\ell(3)=1$ and
       $\ell(4)=0$: $x_3\tanh(x_3)=x_4\coth(x_4)$.}
     \label{figure:curveCH}
   \end{figure}
   \begin{rem}
       Let   $d=4$,  $a_1=a_2=a_3=a_4=1$, $p=2$ and
       $\tau_1=(1,2)$.  In this case, Figure \ref{figure:curveCT} shows the
       intersections of the four  different curves $C_T$ with the boxes
       $I_{2\bn+\bm}\subset\R^2$ for $\bn=(12,2)$ and $\bm\in\{0,1\}^2$. The
       corresponding curve $C_H$ for the particular choice of
      the  hyperbolic factor given by $\ell(3)=1$ and $\ell(4)=0$,   is
       shown on Figure \ref{figure:curveCH}. On each
       of these curves, the marked point  corresponds to the solution of the compatibility equations. 
Note that the curves $C_T$ intersect two of the boxes,  and the functions $T_{a_i,\ell(i)}$ defined on them are positive in one box and negative in the other. 
The solutions of the compatibility equations  lie on the positive side.
     \end{rem}

  We now turn to assertion (iii). Observe first that there is a uniform bound on $c_\ell$ hence there is a $N$ such that if $|\bn| > N$ then
  \[
   \inf_{\bx \in I_\bn}|\bx| > c_\ell.
  \]
From the previous discussion this ensures that there are exactly $2^q$ solutions in the box $I_\bn$. We proceed in two steps for the more quantitative part of the statement. First, we prove that eigenvalues do take the form \eqref{eq:sigmabnk}, and then we show that for all $k \in \set{1,2\dotsc,2^q}$ the same $\balpha_\bn$ works.
  \subsubsection*{Localisation}
  Fix the restriction $\ell:\tau_2\rightarrow\{0,1\}^q$ for the moment.
  The various choices of trigonometric factors (represented by the
  choice of $\ell:\tau_1\rightarrow\{0,1\}$) gives rises to exactly one solution
  $\balpha_{2\bn+\bm}$ in each of the of the $2^p$ boxes $I_{2\bn+\bm}$, where $\bm$
  runs over all choices of $\bm\in\{0,1\}^p$.
  For each of these $\bm$, 
  the corresponding eigenvalue is
  given by any of the functions appearing in Equation \eqref{eq:curve} evaluated
  on any of the 
  coordinates of $(\balpha_{2\bn+\bm},\bbeta_{2\bn+\bm})\in\R^p\times\R^q$. 
  It follows that for each $j\in\tau_2$, and $\bn\in\N^q$
  \begin{equation*}
    \begin{aligned}
      |\bbeta_\bn|^2 &= \sum_{i \in \tau_2} \beta_{\bn,i}^2 = q \beta_{\bn,j}^2 + \bigo{|\bn|^{-\infty}}.
    \end{aligned}
  \end{equation*}
  Hence for each $j\in\tau_2$,
  $$\beta_{\bn,j}=\frac{|\bbeta_{\bn}|}{\sqrt{q}}+ \bigo{|\bn|^{-\infty}}.$$
  The corresponding eigenvalue is therefore given, for any $j \in \tau_2$, by
  \begin{equation*}
    \sigma_{\bn} = H_{a_j,\ell(j)}(\beta_{\bn,j}) =
    \frac{|\bbeta_\bn|}{\sqrt q} + \bigo{|\bn|^{-\infty}}=
    \frac{|\balpha_\bn|}{\sqrt q} + \bigo{|\bn|^{-\infty}},
  \end{equation*}
  as was announced.

\subsubsection*{Clustering}
If $\ell, \ell':S_d\rightarrow\{0,1\}$ agree on $\tau_1$, it follows from
\begin{equation*}
 H_{a_j,\ell(j)}(x) - H_{a_{j},\ell'(j)}(x) = \bigo{x^{-\infty}}
\end{equation*}
that the corresponding eigenvalues satisfy
\begin{equation*}
  \sigma_{\bn,\ell} - \sigma_{\bn,\ell'} = \bigo{|\bn|^{-\infty}}.
\end{equation*}
The various choices of the restriction $\ell:\tau_2\rightarrow\{0,1\}$
therefore lead to $2^q$ eigenvalues satisfying
$$\sigma_\bn^k=\frac{|\balpha_\bn|}{\sqrt q} +\bigo{|\bn|^{-\infty}}\qquad\mbox{for }k=1,\dotsc,2^q.$$

\subsubsection*{Exceptional eigenvalues}
For $\bn \in \N_0^p \setminus \N^p$ we have that $n_i =0 $
for at least one $i \in \tau_1$. On the interval
$\left(0,\frac{\pi}{2a_i}\right]$, the function $T_{a_i,0}$ is
positive while $T_{a_i,1}$ is negative, hence an eigenvalue can only
correspond to $\ell(i) = 0$. In this case, the range of $T_{a_i,0}$ is
  $\left[0,a_i^{-1}\right)$. A corresponding eigenvalue is therefore
  bounded above by $a_i^{-1}$. There is only a finite number of these, proving assertion (iv).

  This concludes the proof of Theorem \ref{thm:eigenlattice}.
\end{proof}


In the next section we will take up the task of understanding the
asymptotic behavior of the counting function $N(\sigma)$.

\section{Eigenvalue asymptotics}\label{section:asymptotics}

The goal of Section \ref{section:asymptotics}  is to prove Theorem \ref{thm:main}. The plan
is to represent  the counting function
$N(\sigma)$ as a sum of auxiliary counting functions
corresponding to different families of eigenvalues provided by
Theorem \ref{thm:eigenlattice}. Each of those counting functions will be then
investigated using lattice counting techniques.

\subsection{A hierarchy of counting functions}

Let $p\in\{1,2,\dotsc,d-1\}$.
Given $\tau = (\tau_1,\tau_2) \in \CT_{p}$ and
$\ell:S_d\rightarrow\{0,1\}$, define the counting function
$N^{\tau,\ell}:\R\rightarrow\N$ by
$$N^{\tau,\ell}(\sigma)=\#\{j\in\N\,:\,\sigma_j\in S_{\tau,\ell}
\mbox{ and }\sigma_j<\sigma\}.$$
Recall that the bipartition $\tau$ defines the location $\tau_1$ of the
trigonometric factors, and the location $\tau_2$ of the hyperbolic
factors, whereas the function $\ell$ 
distinguishes between $\sin$ and $\cos$ trigonometric factors, and
$\sinh$ and $\cosh$ hyperbolic factors.
We also introduce
\begin{equation}
\label{countingtau}
N^{\tau}(\sigma):=\sum_{\ell:S_d\rightarrow\{0,1\}}N^{\tau,\ell}(\sigma)
\quad\mbox{  and }\quad
N_p(\sigma):= \sum_{\tau \in \CT_p} N^{\tau}(\sigma).
\end{equation}
Since there is only a finite number of eigenfunctions with linear
factors, one has
\begin{equation*}
 N(\sigma) = \sum_{p=1}^{d-1} N_p(\sigma)+\bigo 1.
\end{equation*}
Set $q=d-p$ and let $\partial^q \Omega$ denote the union of $p$-dimensional facets of a cuboid $\Omega$. Our goal is to prove the following asymptotics for $N_p(\sigma)$.
\begin{prop}
\label{Npsigma}
For each $p=1,.\dots, d-1$, we have:

\begin{equation}
\label{eq:Npsigma}
 N_p(\sigma) = \frac{\sqrt{q^p}}{(2\pi)^p}\omega_p\Vol_{p}(\del^q(\Omega))\sigma^p + c_p \Vol_{p-1}(\del^{q+1}\Omega)\sigma^{p-1} + \bigo{\sigma^{\eta_p}},
\end{equation}
where $c_p$ are some explicitly computable constants and 
 \begin{equation*}
  \eta_p = \max\left(p-1-\frac{1}{p},\, p - 2 + \frac{2}{p+1}\right) 
  = \begin{cases}
      2/3 & \text{if } p=2, \\
      p-1-1/p & \text{otherwise}.
     \end{cases}
 \end{equation*}
\end{prop}
We prove Proposition \ref{Npsigma} in subsection \ref{Npsigma:proof}.

\subsection{Quasi-eigenvalues}

In this section, we observe that the clustering of eigenvalues in Theorem \ref{thm:eigenlattice} allows us to simplify the eigenvalue counting problem. Essentially, we will count every cluster as one eigenvalue with a weight equal to the number of eigenvalues in the cluster.

\begin{defi}
  Given $p\in S_d$, $q=d-p$, $\tau\in\CT_p$, $\ell:S_d\rightarrow\{0,1\}$ and
  $\bn\in\N^p$, the number $\frac{|\balpha_\bn|}{\sqrt{q}}$ defined in \eqref{eq:sigmabnk}
  is called a \emph{quasi-eigenvalue of multiplicity $2^{q}$}.
\end{defi}
It is clear from Theorem \ref{thm:eigenlattice} that
\begin{gather}\label{def:quasicounting}
  N(\sigma)=\sum_{p=1}^{d-1}2^{q}\#\left\{\bn\in\N^p\,:\,\frac{|\balpha_\bn|}{\sqrt{q} }<\sigma\right\}+\bigo{1}.
\end{gather}
The factor $2^{q}$ accounts for the clustering of eigenvalues around
the corresponding quasi-eigenvalue. Note that the $\bigo{1}$ error can be absorbed in the error term in \eqref{counting:higher}.
Therefore, in view of  \eqref{def:quasicounting}, for our purposes there is no need to distinguish between 
counting eigenvalues and  quasi-eigenvalues.

\subsection{Eigenfunctions with a single trigonometric factor}
Consider first the case $p=1$.
The choice of $\sin$
or $\cos$ for the trigonometric factor and the choice of
the coordinate corresponding to the trigonometric factor yields $2d$ families of eigenfunctions,  each having $2^{d-1}$ possibilities for the choice of the hyperbolic factor. 
As follows from Theorem \ref{thm:eigenlattice}, each of the $2d$ families contributes  a cluster of $2^{d-1}$ eigenvalues  which correspond to the same quasi-eigenvalue. Therefore, as
was mentioned earlier, this cluster can be counted   for our purposes as a single quasi-eigenvalue of multiplicity $2^{d-1}$.
The compatibility equations
  \begin{equation}
\label{compeq}
H_{a_i,\ell(i)}(\beta_i)=H_{a_j,\ell(j)}(\beta_j)\quad\forall
  i,j\in\tau_2
\end{equation}
 define a
  connected curve in $\R^{d-1}$ which goes to infinity
  along the diagonal while its value increases to $+\infty$.
Equating \eqref{compeq}  to $T_{a_k, \ell(k)}$, $k \in \tau_1$ amounts to
solving the following equations:
\begin{equation*}
 \alpha_k \cot(a_k \alpha_k) = \frac{\alpha_k}{\sqrt{d-1}} +
 \bigo{\alpha_k^{- \infty}}
 \qquad\mbox{ if }\ell(k)=0,
\end{equation*}
and
\begin{equation*}
- \alpha_k \tan(a_k \alpha_k) = \frac{\alpha_k}{\sqrt{d-1}}
+\bigo{\alpha_k^{- \infty}}
 \qquad\mbox{ if }\ell(k)=1.
\end{equation*}
This yields eigenvalues of the form 
\begin{equation*}
 \sigma= \begin{cases}
   \frac{\pi j}{a_j\sqrt{d-1}} + \frac{1}{a_j\sqrt{d-1}} \arccot((d-1)^{-1/2}) + \bigo{j^{-\infty}} & \text{if } \ell(k) = 0,\\
   \frac{\pi j}{a_j\sqrt{d-1}} + \frac{1}{a_j\sqrt{d-1}}\arctan((d-1)^{-1/2}) + \bigo{j^{-\infty}} &\text{if } \ell(k) = 1,
 \end{cases}
\end{equation*}
each with quasi-multiplicity $2^{d-1}$.
Given that  $\arccot$ and $\arctan$ are bounded functions, and since 
\begin{equation*}
\Vol_{1}(\del^{d-1}\Omega) = 2^d \sum_{j=1}^d a_j,
\end{equation*}
we have that
\begin{equation*}
 N_1(\sigma) = \frac{\omega_1 \sqrt{d-1}}{2\pi} \Vol_1(\del^{d-1}\Omega)\, \sigma + \bigo{1}.
\end{equation*}
This concludes the proof of Theorem  \ref{thm:main} for $d = 2$, since $p = 1$ is the only possibility in this case. Observe that for $d = 2$, this is indeed the expected first term of Weyl's law \eqref{C1Weyl}.

\subsection{Eigenfunctions with many trigonometric factors}
\label{many}
In this subsection, we count the number of eigenvalues associated with
eigenfunctions with more than one trigonometric factor. The idea is to
write the eigenvalues as the norms of points $\balpha \in \R^p$ that are
close to some lattice points. The main difficulty
is that the compatibility equations are transcendental, making it
impossible to explicitly find $\balpha$. We will therefore approximate
the eigenvalues in a controlled way, and we will show that this
approximation results in a small enough error that could be absorbed
in the remainder in the two-term asymptotics for the eigenvalue
counting function. Finally, we will use  the lattice point counting
techniques going back to  \cite{hlawka, randol}, and more recently
used in \cite{lagaceparnovski}.

\subsubsection{Approximate eigenvalues}
Suppose that  $d\geq 3$ and
  $p\in\{2,\dotsc,d-\krn1\}$. Let $\tau\in\CT_p$ and
$\ell:S_d\rightarrow\{0,1\}$ be given.

Given $\bn\in\N^p$, it follows from Theorem
\ref{thm:eigenlattice}  and the compatibility equations
\eqref{equation:compatibility},  that the corresponding solution
$\balpha=\balpha_\bn\in I_\bn$ satisfies the following for each $i,j\in\tau_1$
\begin{equation*}
 \alpha_i \cot\left(\alpha_i a_i + \frac{\ell(i) \pi}{2}\right) = \alpha_j \cot\left(\alpha_j a_j + \frac{\ell(j) \pi}{2}\right) =
 \frac{|\balpha_\bn|}{\sqrt q} + \bigo{|\bn|^{-\infty}}.
\end{equation*}
Hence,  for each $i\in\tau_1$, we have, choosing the principal branch of $\arccot$, a family of solutions indexed by $\bn \in \N^p$
\begin{equation*}
 \alpha_i a_i = \left( n_i+\frac{\ell(i)}{2}\right)\pi  + \arccot\left(\frac{1}{\sqrt q}\left[1 + \sum_{j \ne i\in\tau_1}
     \left(\frac{\alpha_j}{\alpha_i}\right)^2\right]^{1/2}\right) + \bigo{|\bn|^{-\infty}}. 
\end{equation*}
Since $\alpha_i= \frac{\left(n_i + \frac{\ell(i)}{2}\right) \pi}{a_i}+O(1)$, we can rewrite the previous equation as follows 
\begin{equation} \label{eq:transcendental}
 \begin{aligned}
 \alpha_i &= \frac{\left(n_i+\frac{\ell(i)}{2}\right)\pi}{a_i} \\ &\quad+ \frac{1}{a_i} \arccot\left(\frac{1}{\sqrt q}\left[1 + \sum_{j \ne i} \left(\frac{\frac{\left(n_j + \frac{\ell(j)}{2}\right) \pi}{a_j} + t_{\alpha_j}(\bn)}{\frac{\left(n_i + \frac{\ell(i)}{2}\right)\pi}{a_i} + t_{\alpha_i}(\bn)}\right)^2\right]^{1/2}\right)+\bigo{|\bn|^{-\infty}},
\end{aligned}
 \end{equation}
where the functions $t_{\alpha_j}$ are bounded. Since $\ell(i)$ ranges over $\set{0,1}$, the solution set to the previous equation is the same as the one to
\begin{equation}
  \alpha_i = \frac{n_i\pi}{2a_i} + \frac{1}{a_i} \arccot\left(\frac{1}{\sqrt q}\left[1 + \sum_{j \ne i} \left(\frac{\frac{n_j \pi}{2a_j} + t_{\alpha_j}(\bn)}{\frac{n_i\pi}{2a_i} + t_{\alpha_i}(\bn)}\right)^2\right]^{1/2}\right)+\bigo{|\bn|^{-\infty}}.
\end{equation}
\begin{lem} \label{lem:alphatilde}
Define $\tilde \alpha_i$ as 
\begin{equation} \label{eq:analytic}
 \tilde \alpha_i = \frac{ n_i\pi}{2a_i} + \frac{1}{a_i}  \arccot\left(\frac{1}{\sqrt q}\left[1 + \sum_{j \ne i} \left(\frac{a_i n_j}{a_j n_i}\right)^2\right]^{1/2}\right).
\end{equation}
Then,
\begin{equation} \label{eq:approximatealpha}
 \tilde \alpha_i  = \alpha_i + \bigo{|\bn|^{-1}}
\end{equation}
\end{lem}
\begin{proof}
In Lemma \ref{lem:arccot} in the Appendix, take $x_i = \frac{n_i \pi}{a_i}$ and $\psi_i = t_{\alpha_i}$. Then, one readily sees that 
\[
 |\bx|\asymp |\bn|,
\]
where $f \asymp g$ means that $f = \bigo g$ and $g = \bigo f$. The lemma then follows. 
\end{proof}
Note that the right hand side of equation \eqref{eq:analytic} does not depend on $\alpha_i$ anymore, which makes it easier to analyse.

We now have eigenvalues indexed by $\bn \in \N^p$ given by
\begin{equation} \label{eq:approxev}
 \sigma_\bn = \sqrt{\frac{1}{ q} \sum_{i \in \tau_1} \tilde \alpha_i^2} + \bigo{|\bn|^{-1}}.
\end{equation}

\begin{defi}
The numbers
 \begin{equation} \label{eq:approxsigma}
  \tilde \sigma_\bn = \sqrt{\frac{1}{ q} \sum_{i \in \tau_1} \tilde \alpha_i^2}
 \end{equation}
are called the {\it approximate eigenvalues}. 
\end{defi}

\begin{rem}
 Up until now, eigenvalues, quasi-eigenvalues and approximate eigenvalues were indexed by $\bn \in \N^p$. In the following two theorems it is convenient to use $n \in \N$ to index them in an ascending order.
\end{rem}

The following lemma allows us to estimate the error induced by counting approximate eigenvalues  instead of eigenvalues. 
\begin{lem} \label{lem:closeev}
  Let $(a_n), (b_n)$ be two sequences of positive numbers
  which tend to infinity. Suppose there exists  a number $s > -1$ such that 
  $a_n = b_n + \bigo{b_n^{-s}}$. Let
 \begin{equation*}
  N_a(\lambda) = \#\set{n : a_n < \lambda} 
\qquad\mbox{ and }\qquad
 N_b(\lambda) = \#\set{n : b_n < \lambda}.
\end{equation*}
Suppose that there exists  a number $K$ such that
\begin{equation*}
N_a(\lambda) = \sum_{k=0}^K c_k \lambda^{p - k} + \bigo{\lambda^{r}},
\end{equation*}
with $r < p-K$. Then,
\begin{equation} \label{eq:closecountingfunction}
 N_b(\lambda) = \sum_{k=0}^K c_k \lambda^{p - k} + \bigo{\lambda^{r'}}
\end{equation}
where $r' = \max(r,p-1-s)$.
\end{lem}
\begin{rem}
Note that if $r' \ge p - K$, some of the terms in the sum in \eqref{eq:closecountingfunction} might be absorbed in the error term. 
\end{rem}

\begin{proof}
Indeed, the assumption on the sequences $a_n$ and $b_n$ implies that there exists $c > 0$ such that
 \begin{equation*}
  N_a\left(\lambda + \frac{c}{\lambda^s}\right) \le N_b(\lambda) \le N_a\left(\lambda - \frac{c}{\lambda^s}\right).
 \end{equation*}
 A direct computation of $N_a(\lambda \pm c \lambda^{-s})$ completes the proof of the lemma.
\end{proof}

 Recall now the definition of  $N^\tau(\sigma)$ given by \eqref{countingtau}. We will write $\tilde N^\tau$ for the counting function of the corresponding approximate eigenvalues.  \begin{lem}
We have:
\begin{equation*}
 \left|\tilde N^{\tau}(\sigma) - N^{\tau}(\sigma)\right| = \bigo{\sigma^{p - 1 - 1/p}}.
\end{equation*}
 \end{lem}

 \begin{proof} Both the eigenvalues and the approximate eigenvalues are, up to a bounded error, the norms of the points of the lattice $\Gamma = \bigoplus_{i=1}^p \frac{\pi}{2 a_i\sqrt q}\N$, repeated $2^q$ times. 
 Denote  by $l_n := \set{|\bgamma| : \bgamma \in \Gamma}_n$ the sequence of norms of the points of the  lattice $\Gamma$ arranged in ascending order. It is well known that there is a constant $C$ such that
 \begin{equation*}
  N_l(\sigma) = C \sigma^p + \bigo{\sigma^{p-1}},
 \end{equation*}
 where $C$ depends on $\Gamma$ and $N_l$ denotes the counting function of the sequence $l_n$ as in Lemma \ref{lem:closeev}. 
Applying Lemma \ref{lem:closeev} with $s = 0$ yields
\begin{equation*}
 N^\tau(\sigma) = 2^q C \sigma^p + \bigo{\sigma^{p-1}}.
\end{equation*}
Reversing this expression tells us that
\begin{equation} \label{eq:asympsigma}
 \sigma_n = \left(\frac{n}{2^q C} \right)^{1/p}+\smallo{n^{1/p}}.
\end{equation}

From equations \eqref{eq:approxsigma} and \eqref{eq:asympsigma} we have that
\begin{equation*}
 \tilde\sigma_n = \sigma_n + \bigo{n^{-1/p}}.
\end{equation*}
Therefore, applying once again Lemma \ref{lem:closeev}, but this time with $s = 1/p$, yields
\begin{equation} \label{eq:Ntilde}
 N^{\tau}(\sigma) = \tilde N^{\tau}(\sigma) + \bigo{\sigma^{p - 1 - 1/p}}.
\end{equation}
\end{proof}

\subsubsection{Another representation of the counting function}

For every $\tau$, let us now  define a  family of sets $E_\sigma \subset \R^p$ with the property that
\begin{equation} \label{eq:goalindicator}
\tilde N^\tau(\sigma) = \sum_{\bn \in \N^p} 2^q \chi\left(\frac{\bn}{\sigma}\right) + \bigo 1,
\end{equation}
where $\chi:= \chi_\sigma$ is the indicator function of $E_\sigma$. Let us define elliptic polar coordinates in $\R^p$ with the convention that $\theta_p=0$ :
 \begin{equation}
 \label{eq:ellipticpolar}
  \begin{aligned}
   r^2 &= \sum_{i \in \tau_1} \left(\frac{\pi x_i}{2a_i \sqrt q}\right)^2, \\
   x_j &= r \frac{2a_j \sqrt q}{\pi}\cos(\theta_j) \prod_{i < j} \sin(\theta_i).
  \end{aligned}
 \end{equation}
We define the family of sets
\begin{equation} \label{eq:Esigma}
 E_\sigma:= \set{(r,\btheta) \in \R^p : r^2 + \frac{2r}{\sigma}\sum_{j \in \tau_1}\frac{1}{a_j} g_j(\btheta) + \frac{H(\btheta)}{\sigma^2} < 1},
\end{equation}
with
\begin{equation} \label{eq:gj}
 g_j(\btheta) :=  \cos\theta_j \prod_{i < j} \sin\theta_i \arccot\left(\frac{1}{\sqrt q}\left[1 + \sum_{i \ne j} \left(\frac{x_i}{x_j}\right)^2\right]^{1/2}\right),
\end{equation}
and
\begin{equation*}
H=H(\btheta) = \sum_{j \in \tau_1} \frac{1}{a_j^2} \arccot\left(\frac{1}{\sqrt q} \left[1 + \sum_{i \ne j} \left(\frac{x_i}{x_j}\right)^2\right]^{1/2}\right)^2.
\end{equation*}
 From equation \eqref{eq:approxsigma}, we can observe that the evaluation of $\chi$ at $\sigma^{-1} \bn$ in coordinates \eqref{eq:ellipticpolar} is $1$ if and only if $\tilde \sigma_\bn < \sigma$. If $|\bn| > N$ as in Theorem \ref{thm:eigenlattice}, there are $2^q$ solutions close to any order to $\tilde \sigma_\bn$. This achieves our stated goal of equation \eqref{eq:goalindicator}. Let us now prove a few properties of the set $E_\sigma$ that will be required in the sequel.
 \begin{lem} \label{lem:curvature}
  There exists $\sigma_0$ such that for $\sigma> \sigma_0$ the set $E_\sigma$ is strictly convex and the principal curvatures of $\del E_\sigma$ are positive and uniformly bounded away from $0$. Furthermore, all the derivatives of the principal curvatures tend to $0$ as $\sigma \to \infty$.
 \end{lem}

 \begin{proof}  From equation $\eqref{eq:Esigma}$ $\del E_\sigma$ is the level set of a function $F$ satisfying
\begin{equation} \label{eq:esigma}
\begin{aligned}
 F(r,\btheta) &= r^2 + \bigo{\sigma^{-1}}, \\
\left[\nabla F(\bx)\right]_i &= \frac{\pi x_i}{a_i \sqrt{q}} + \bigo{\sigma^{-1}}, \\
\operatorname{Hess} F &= \operatorname{diag}\left(\frac{\pi}{a_i \sqrt q}\right)_{i \in \tau_1} + \bigo{\sigma^{-1}},
 \end{aligned}
\end{equation}
 with the error estimates uniform in $\del E_\sigma$. This yields that for $\sigma$ large enough, the second fundamental form of $\del E_\sigma$ is positive, with its smallest eigenvalue uniformly bounded away from $0$. This implies the claim on the principal curvatures, which in turn implies strict convexity. 
 
 As for the derivatives of the principal curvatures, they are the derivatives of the eigenvalues of $\operatorname{Hess} F$. Observe that $r g_j$ and $H$ are smooth away from the origin, hence all their derivatives are bounded on $\del E_\sigma$. This implies that the derivatives of $\operatorname{Hess} F$ are $\bigo{\sigma^{-1}}$, hence the derivatives of its eigenvalues as well and they go to $0$ as $\sigma \to \infty$. 
 \end{proof}
 This argument also yields the following corollary.
 \begin{cor}
  The product of the principal curvatures of $\del E_\sigma$ is uniformly bounded away from zero for $\sigma$ large enough.
 \end{cor}

 \subsubsection{Poisson Summation Formula}
 In this section, we use the general scheme of the proof of \cite[Theorem 1.1]{lagaceparnovski}.
Recall that
\begin{equation} \label{eq:recall}
\begin{aligned}
 N^{\tau}(\sigma) &= \sum_{\bn \in \N^p} 2^q \chi\left(\frac \bn \sigma\right) + \bigo 1\\
 &= 2^{q - p}\sum_{\bn \in \Z^p} \chi\left(\frac{\bn}{\sigma} \right) + R_{\tau}(\sigma) + \bigo 1,
 \end{aligned}
\end{equation}
where $R_{\tau}(\sigma)$ is the error term induced by the overcounting of
points on hyperplanes with one vanishing coordinate.

Our goal is now to compute the terms appearing in equation \eqref{eq:recall} using the Poisson summation formula which states, under sufficient smoothness assumptions that
\begin{equation} \label{eq:psf}
 \sum_{\bn \in \Z^p} f(\bn) = \sum_{\bm \in \Z^p} \hat f(\bm)
\end{equation}
where the Fourier transform is given by
\begin{equation*}
 \hat f(\bxi) := \int_{\R^p} f(\bx) e^{-2 \pi i \bx \cdot \bxi} \de \bx.
\end{equation*}

However, $\chi$ is not regular  enough for us to use the Poisson summation
formula, hence we need to mollify it. Let us introduce a nonnegative function $\psi \in
C_c^\infty(\R)$ supported in $[-1,1]$ and such that 
\begin{equation*}
 \int_0^\infty \psi(r) r^{p-1} \de r = \frac{1}{V_{p-1}},
\end{equation*}
with $V_{p-1}$ being the volume of the $p-1$ dimensional unit sphere in $\R^p$. We then define a family $\Psi_\epsilon : \R^p \to \R$ of radial bump functions of total mass $1$ by
\begin{equation*}
 \Psi_\epsilon(\bx) = \frac{1}{\epsilon^p}\psi\left(\frac{|\bx|}{\epsilon}\right).
\end{equation*}
Set $\Psi := \Psi_1$ Consider the smooth function $\chi_\epsilon = \Psi_\epsilon * \chi$. Note that 
\begin{equation*}
 \hat \Psi_\epsilon(\bxi) = \hat \Psi(\epsilon \bxi) 
\end{equation*}
We now prove the following lemma.
\begin{lem}
 Let $\chi_\epsilon^+, \chi_\epsilon^- : \R^p \to \R$ be defined by
 \begin{equation*}
 \begin{aligned}
  \chi_\epsilon^+(\bx) &= \chi_\epsilon\left((1-\eta_+\epsilon)\bx\right) \\
  \chi_\epsilon^-(\bx) &= \chi_\epsilon\left((1+\eta_-\epsilon)\bx\right)
  \end{aligned}
 \end{equation*}
 for some $\eta_-,\eta_+>0$. One can choose $\eta_-,\eta_+$ in such a way that for all $\sigma$ large enough
 \begin{equation*}
  \chi_\epsilon^-(\bx) \le \chi(\bx) \le \chi_\epsilon^+(\bx)
 \end{equation*}
for all $\bx \in \R^p$ and all $\epsilon>0$ small enough.
\end{lem}
\begin{proof}
 For the first inequality, observe that
 \begin{equation*}
  \begin{aligned}
   \chi_\epsilon((1 + \eta_- \epsilon)\bx) &= \int_{\R^p} \chi(\by) \Psi_\epsilon((1 + \eta_- \epsilon)\bx - \by) \de \by \\
   &= \int_{B_{(1 + \eta_- \epsilon) \bx}(\epsilon)} \chi(\by) \Psi_\epsilon((1 + \eta_- \epsilon)\bx - \by) \de \by\\
   &\le \sup_{B_{(1 + \eta_- \epsilon) \bx}(\epsilon)} \chi(\by).
  \end{aligned}
 \end{equation*}
 Hence, to show that $\chi_\epsilon((1 + \eta_-\epsilon)\bx) \le
 \chi(\bx)$ for all $\bx$, by convexity of $E_\sigma$ it is sufficient to
 show that for all $\bx \in \del E_\sigma$, there exists $\eta_-$,
 independent of $\sigma$ such that the following holds for each
 $\epsilon>0$ small enough
\begin{equation*}
 B_{(1 + \eta_- \epsilon)\bx}(\epsilon) \cap E_\sigma = \varnothing.
\end{equation*}
Note that for all $\bx \in \del E_\sigma$, we have that
\begin{equation} \label{eq:distboundary}
 \operatorname{dist}((1 + t) \bx, \del E_\sigma) = (\bx \cdot \mathcal N_{\del E_\sigma}(\bx)) t + \bigo{t^2},
\end{equation}
where $\mathcal N_{\del E_\sigma}$ is the Gauss map of the boundary. To see this, denote by $T_{\bx}\del E_\sigma$ the tangent hyperplane of $\del E_\sigma$ at $\bx$, and by $P_\bx$ the orthogonal projection on that hyperplane. We have by the triangle inequality that
\[
 \left|\operatorname{dist}((1 + t) \bx, \del E_\sigma) - \operatorname{dist}((1 + t) \bx, T_\bx \del E_\sigma) \right| \le  \operatorname{dist}(P_\bx((1+t)\bx),\del E_\sigma).
\]
We observe that $\operatorname{dist}((1 + t) \bx, T_\bx \del E_\sigma) = (\bx \cdot \mathcal N_{\del E_\sigma}(\bx)) t$. Let $F$, as before, be the function in $\R^p$ such that the set $F \equiv 1$ coincides with $\partial E_\sigma$.  
Taking the  Taylor expansion of $F$ around  $\bx$,  we have that
\[
 \operatorname{dist}(P_\bx((1+t)\bx),\del E_\sigma) \le \|\operatorname{Hess} F(\bx) \|_{\infty} |P_\bx((1 + t)\bx)|^2 = \bigo{t^2},
\]
where we used that $\|\operatorname{Hess} F(\bx) \|_{\infty}$ is bounded uniformly for $\sigma > \sigma_0$ and $\bx \in \del E_\sigma$. Note that the strict convexity of $\del E_\sigma$ and equation \eqref{eq:esigma} imply that
$\bx \cdot \mathcal N_{\del E_\sigma}(\bx)$ is bounded away from zero uniformly for $\sigma > \sigma_0$. This implies that we can
choose $\eta_-$ large enough and independent in $\sigma$ such that
indeed
\begin{equation*}
 B_{(1 + \eta_- \epsilon)\bx}(\epsilon) \cap E_\sigma = \varnothing.
\end{equation*}
For the second inequality, we have
  \begin{equation*}
  \begin{aligned}
   \chi_\epsilon((1 - \eta_+ \epsilon)\bx) &= \int_{\R^p} \chi(\by) \Psi_\epsilon((1 - \eta_+ \epsilon)\bx - \by) \de \by \\
   &= \int_{B_{(1 - \eta_+ \epsilon) \bx}(\epsilon)} \chi(\by) \Psi_\epsilon((1 - \eta_+ \epsilon)\bx - \by) \de \by\\
   &\ge \inf_{B_{(1 - \eta_+ \epsilon) \bx}(\epsilon)} \chi(\by).
  \end{aligned}
 \end{equation*}
Hence, to show that $\chi(\bx) \le \chi_\epsilon((1 - \eta_+)\bx)$, it is sufficient to show that for all $\bx \in \del E_\sigma$, there exists $\eta_+$ independent of $\sigma$ such that
\begin{equation*}
 B_{(1 - \eta_+)\bx}(\epsilon) \subset E_\sigma.
\end{equation*}
Using once again equation \eqref{eq:distboundary} and arguing exactly as above yields the desired  number $\eta_+$.
\end{proof}
The following is an immediate corollary of the previous lemma:
\begin{cor}
 We have that
 \begin{equation*}
  \sum_{\bn \in \Z^p} \chi_\epsilon^-\left(\frac \bn \sigma \right) \le  \sum_{\bn \in \Z^p} \chi\left(\frac \bn \sigma \right) \le \sum_{\bn \in \Z^p} \chi_\epsilon^+\left(\frac \bn \sigma \right).
 \end{equation*}
\end{cor}

We will now apply the Poisson summation formula \eqref{eq:psf} to $\chi_\epsilon^\pm$, which are smooth functions. This yields, using the basic properties of the Fourier transform,
\begin{equation} \label{eq:poisson}
 \begin{aligned}
  \sum_{\bn \in \Z^p} \chi_\epsilon^\pm \left(\frac{\bn}{\sigma}\right) &= \sigma^p \sum_{\bm \in \Z^p} \hat \chi_\epsilon^\pm(\sigma \bm) \\
  &= \sigma^p \sum_{\bm \in \Z^p} \left(1 + \bigo{\epsilon}\right) \hat
  \chi\left(\frac{\sigma \bm}{1 \mp \eta_\pm \epsilon}\right)
  \hat\Psi\left(\frac{\epsilon \bm \sigma}{1\mp \eta_\pm\epsilon}\right) \\
  &= \sigma^p\Vol(E_\sigma) + \bigo{\epsilon \sigma^p}  \\ & \qquad+ \, \bigo{\sum_{\substack{\bm \in \Z^p \\ \bm \ne 0}} \sigma^p\hat \chi\left(\frac{\sigma \bm}{1 \mp \eta_\pm \epsilon}\right) \hat\Psi\left(\frac{\epsilon \bm \sigma}{1\mp \eta_\pm\epsilon}\right)}.
 \end{aligned}
\end{equation}
Note that for this expression to hold, we will need to later choose $\epsilon = \smallo 1$. Since $\Psi$ is a Schwartz function, its Fourier transform is also Schwartz, hence to find estimates on the asymptotic behaviour of equation \eqref{eq:poisson}, we only need to find bounds on $\hat \chi$. This is done in the following Lemma.

\begin{lem}
 For $\sigma$ large enough, the Fourier transform of $\chi$ satisfies  the upper bound
 \begin{equation}\label{equation:fourier}
  \begin{aligned}
 \hat \chi (\bxi) &= \bigo{|\bxi|^{-\frac{d+1}{2}}}.
 \end{aligned}
\end{equation}
\end{lem}
\begin{proof}
For
$\sigma$ large enough, the set $E_\sigma$ is strictly convex and has
smooth boundary. Therefore, following \cite[Theorem 2.29]{IL} we have
that for any function $f \in C^\infty(\R^p)$ such that $f \ne 0$ on
$\del E_\sigma$,
\begin{equation*}
\begin{aligned}
 \int_{E_\sigma} f(\bx) e^{-2\pi i \bx \cdot \bxi} \de \bx &= \bigo{|\bxi|^{-\frac{d+1}{2}}},
 \end{aligned}
\end{equation*}
where the implicit constants depend on the product of the principal curvatures of $\del
E_\sigma$ and stay bounded as long as the principal curvatures are bounded away from $0$. Hence, by equation \eqref{eq:esigma}, these constants will
be uniformly bounded for $\sigma$ large enough. Applying this result with
$f(\bx)\equiv 1$ yields the desired result.
\end{proof}
\begin{rem}
 Note that the estimates and the error terms obtained in \cite[Theorem 2.29]{IL} depend on the bounds on the  derivatives of the principal curvatures. By Lemma \ref{lem:curvature} the derivatives of the principal curvatures of $\del E_\sigma$ 
 tend to zero as $\sigma \to \infty$, and therefore they could be bounded uniformly for $\sigma > \sigma_0$.
\end{rem}

We now find the dependence on $\epsilon$ of the third summand in $\eqref{eq:poisson}$. We will choose the optimal value of $\epsilon$ such that the second and the third terms are both as small as possible. Splitting the third summand into two terms we use equation
\eqref{equation:fourier} and the fact that $\hat\Psi$ is a Schwartz
function to obtain
\begin{equation*}
\begin{aligned}
 \bigo{\sum_{\substack{\bm \in \Z^p \\ m \ne 0}} \sigma^p \hat \chi\left(\frac{\bm \sigma}{1\mp \eta_\pm\epsilon}\right) \hat\Psi\left(\frac{\epsilon \bm \sigma}{1\mp \eta_\pm\epsilon}\right)} = O &\left(\sum_{0 < |\bm| \le (\epsilon \sigma)^{-1}} \frac{\sigma^\frac{p-1}{2}\left(1 \mp \eta_\pm \epsilon\right)^{\frac{p+1}{2}}}{|\bm|^{\frac{p+1}{2}}}\right. \\ &\qquad + \left.\sum_{|\bm|> (\epsilon \sigma)^{-1}} \frac{\sigma^{\frac{p-1}{2}}\left(1 \mp \eta_\pm \epsilon\right)^{\frac{p+1}{2}+N}}{ |\bm|^{\frac{p+1}2 + N} (\sigma \epsilon)^N}\right),
\end{aligned}
 \end{equation*}
for an arbitrary $N>0$ which will be fixed below.
Assuming that $\epsilon$ is small and  and taking into account that the summands on the right hand side are decreasing in $|\bm|$, we may estimate the  first of those sums  by
\begin{equation*}
\begin{aligned}
 \sum_{0 < |\bm| \le (\epsilon \sigma)^{-1}} \frac{\sigma^{\frac{p-1}{2}}\left(1 \mp \eta_\pm \epsilon\right)^{\frac{p+1}{2}}}{|\bm|^{\frac{p+1}{2}}} &\asymp \sigma^{\frac{p-1}{2}} \int_1^{(\epsilon \sigma)^{-1}} \frac{r^{p-1}}{r^{\frac{p+1}2 }} \de r \\
 &= \bigo{\epsilon^{\frac{1-p}2}}.
 \end{aligned}
\end{equation*}
The second of those sums can be estimated, for $N$ large enough that the integral converges, by 
\begin{equation*}
\begin{aligned}
\sum_{|\bm|> (\epsilon \sigma)^{-1}} \frac{\sigma^{\frac{p-1}{2}}\left(1 \mp \eta_\pm \epsilon\right)^{\frac{p+1}{2}+N}}{|\bm|^{\frac{p+1}2 + N} (\sigma \epsilon)^N} &\asymp \sigma^{\frac{p-1}{2}}(\sigma\epsilon)^{-N} \int_{(\epsilon \sigma)^{-1}}^\infty \frac{r^{p-1}}{r^{\frac{p+1}2 + N }} \de r \\
 &=  \bigo{\epsilon^{\frac{1-p}2}}.
 \end{aligned}
\end{equation*}

The optimal $\epsilon$ to make both $\sigma^p \epsilon$ and $\epsilon^{\frac{1-p}{2}}$ as small as possible is 
\begin{equation*}
 \epsilon = \sigma^{\frac{-2p}{1 + p}},
\end{equation*}
yielding that
\begin{equation} \label{eq:psferror}
 \sum_{n \in \Z^p} \chi_\epsilon^\pm \left(\frac{n}{\sigma}\right) = \sigma^p\Vol(E_\sigma) + \bigo{\sigma^{p - 2 + \frac{2}{1+p}}}.
\end{equation}

We now compute the volume of $E_\sigma$.

\begin{lem} \label{lem:volumeesigma}  Let $\Sigma = \S^{p-1} \cap \R^p_+$. We have:
\begin{equation} \label{eq:volumeesigma}
  \Vol_p(E_\sigma) = \frac{2^p\sqrt{q}^p}{\pi^p}\omega_p \prod_{j \in \tau_1} a_j - \frac{2^{2p} \sqrt{q}^p  G_{p,q}}{\pi^p\sigma} \sum_{j \in \tau_1} \prod_{i \ne j} a_i + \bigo{\sigma^{-2}},
\end{equation}
where
\begin{equation} \label{eq:Gpq}
  G_{p,q} = \int_\Sigma g_j(\btheta) \de \btheta,
\end{equation}
for any of the functions $g_j$ defined by equation \eqref{eq:gj}.
\end{lem}
\begin{rem}
 Note that $G_{p,q}$ does not depend on $j$ by the symmetry of the construction of $g_j$.
\end{rem}

\begin{proof}
  By symmetry, we have that
\begin{equation*}
 \begin{aligned}
  \Vol(E_\sigma) = \frac{2^{2p}\sqrt q^p}{\pi^p} \int_\Sigma \int_0^{\rho(\btheta)} r^{p-1} \prod_{j \in \tau_1} a_j \de r \de \btheta
 \end{aligned}
\end{equation*}
where $\rho(\btheta)$ is the unique positive root (in $r$) of the equation
\begin{equation*}
 r^2 + \frac{2r}{\sigma} \sum_{j \in \tau_1} \frac{g_j(\btheta)}{a_j} + \frac{H}{\sigma^2} - 1 = 0.
\end{equation*}
One can observe that
\begin{equation*}
 \rho(\btheta) = 1 - \frac{1}{\sigma} \sum_{j \in \tau_1} \frac{g_j(\btheta)}{a_j} + \bigo{\sigma^{-2}}.
\end{equation*}
Thus, we get that
\begin{equation*}
 \Vol(E_\sigma) = \frac{2^{2p}}{\pi^p} \prod_{j \in \tau_1} a_j \int_\Sigma \frac 1 p - \frac{1}{\sigma} \sum_{j \in \tau_1} \frac{g_j(\btheta)}{a_j} + \bigo{\sigma^{-2}} \de \btheta
\end{equation*}
Integrating and replacing in the previous equation the definition of $G_{p,q}$ in equation \eqref{eq:Gpq} yields
\begin{equation} \label{eq:volesigma}
 \Vol(E_\sigma) = \frac{2^p}{\pi^p}\omega_p \prod_{j \in \tau_1} a_j - \frac{2^{2p}  G_{p,q}}{\pi^p\sigma} \sum_{j \in \tau_1} \prod_{i \ne j} a_i + \bigo{\sigma^{-2}}.
\end{equation}
\end{proof}
Finally, we have to take into account the points that we have overcounted
with coefficient $1/2$ on the hyperplanes $\set{x_i = 0}$. This is given in the following lemma.
\begin{lem} \label{lem:overcounted}
 The number of overcounted points on the hyperplanes $\set{x_i = 0}$ is 
 \begin{equation} \label{eq:overcounted}
  R_{\tau}(\sigma) = \frac{\sqrt{q^p}2^{p} \omega_{p-1} \sigma^{p-1}}{4 (2\pi)^{p-1}}\sum_{j \in \tau_1} \prod_{i \ne j} a_i + \bigo{\sigma^{p-2}}.
 \end{equation}

\end{lem}

\begin{proof}
One can observe that $R_{\tau}$ is given by
\begin{equation*}
R_{\tau}(\sigma) =  \frac 1 2 \sum_{i \in \tau_1} \#\set{\sigma^{-1}\N^{p-1} \cap E_\sigma \cap \set{x_i = 0}}
\end{equation*}
Since $E_\sigma$ is convex, rough lattice point counting estimates due to Gauss tell us that
\begin{equation*}
 R_{\tau}(\sigma) = \frac 1 2 \sigma^{p-1} \sum_{i \in \tau_1}\Vol_{p-1}\left(E_\sigma \cap \set{x_i = 0}\right) + \bigo{\sigma^{p-2}}.
\end{equation*}
Computing the volumes in the same way as in the proof of the previous lemma yields the desired result.
\end{proof}

\subsection{Proof of Proposition \ref{Npsigma}.}
\label{Npsigma:proof}
Recall that $\tilde N_p$ is given by
\begin{equation*}
 \tilde N_p(\sigma) = \sum_{\tau \in \CT_p} \tilde N^\tau(\sigma).
\end{equation*}
Observe that 
\begin{equation*}
 \sum_{\tau \in \CT_p} 2^{p+q} \prod_{j \in \tau_1} a_j = \Vol_{p}(\del^q(\Omega))
\end{equation*}
and 
\begin{equation}\label{eq:p-1volume}
\begin{aligned}
 \sum_{\tau \in \CT_p} \sum_{j \in \tau_1} \prod_{i \ne j}2^{p+q} a_i &= (q+1)2^{p+q} \sum_{\tau \in \CT_{p-1}} \prod_{j \in \tau_1} a_j \\
 &= (q+1) \Vol_{p-1}(\del^{q+1}\Omega)).
 \end{aligned}
\end{equation}
Combining these two formulas with equations \eqref{eq:recall}, \eqref{eq:psferror} and Lemmas \ref{lem:volumeesigma}, \ref{lem:overcounted}, yields
\begin{equation*}
 \tilde N_p(\sigma) = \frac{\sqrt{q^p}}{(2\pi)^p}\omega_p\Vol_{p}(\del^q(\Omega))\sigma^p + c_p \Vol_{p-1}(\del^{q+1}\Omega)\sigma^{p-1} + \bigo{\sigma^{p - 2 + \frac{2}{p+1}}}.
\end{equation*}
Using equation \eqref{eq:p-1volume}, we have that $c_p = c_p' + c_p''$, where
\[
 c_p' = - \frac{(q+1) \sqrt{q^p}G_{p,q}}{\pi^p}
\]
comes from the second term in equation \eqref{eq:volumeesigma} and
\[
 c_p'' = - \frac{(q+1) \sqrt{q^p} \omega_{p-1}}{4(2\pi)^{p-1}}
\]
is obtained from the principal term in equation \eqref{eq:overcounted}.

We then have from equation \eqref{eq:Ntilde} that
\begin{equation*}
 N_p(\sigma) = \frac{\sqrt{q^p}}{(2\pi)^p}\omega_p\Vol_{p}(\del^q(\Omega))\sigma^p + c_p \Vol_{p-1}(\del^{q+1}\Omega)\sigma^{p-1} + \bigo{\sigma^{\eta_p}},
\end{equation*}
where
 \begin{equation*}
  \begin{aligned}
  \eta_p &= \max\left(p-1-1/p, p - 2 + \frac{2}{p+1}\right) \\
  &= \begin{cases}
      2/3 & \text{if } p=2, \\
      p-1-1/p & \text{otherwise}.
     \end{cases}
 \end{aligned}
 \end{equation*}
This completes the proof of Proposition \ref{Npsigma}. 
\subsection{Proof of Theorem \ref{thm:main}.}
Recall now that
\begin{equation*}
 N(\sigma) = \sum_{p = 1}^{d-1} N_p(\sigma) + \bigo 1.
\end{equation*}
Hence, applying the previous results we get
\begin{equation*}
\begin{aligned}
 N(\sigma) &= N_{d-1}(\sigma) + N_{d-2}(\sigma) + \bigo{\sigma^{\eta_{d-1}}}\\
 &= \frac{1}{(2\pi)^{d-1}}\omega_{d-1}\Vol_{d-1}(\del(\Omega))\sigma^{d-1} + c_{d-1} \Vol_{d-2}(\del^{2}\Omega)\sigma^{d-2}\\
 &\qquad + \frac{2^{\frac{d-2}2}}{(2\pi)^{d-2}}\omega_{d-2}\Vol_{{d-2}}(\del^2(\Omega))\sigma^{d-2} + \bigo{\sigma^{\eta_{d-1}}} \\
 &= C_1 \Vol_{d-1}(\del \Omega) \sigma^{d-1} + C_2 \Vol_{d-2}(\del^2\Omega) \sigma^{d-2} + \bigo{\eta_{d-1}}.
\end{aligned}
 \end{equation*}
We can write explicitly $C_2 = c_{d-1}' + c_{d-1''} + \frac{2^{\frac{d-2}2} \omega_{d-2}}{(2\pi)^{d-2}}$ to get indeed that
\begin{equation*}
C_2 = \frac{2^{\frac{d-2}2}\omega_{d-2}}{(2\pi)^{d-2}} -  \frac{2 G_{d-1,1}}{\pi^{d-1}} -  \frac{ \omega_{d-2}}{2(2\pi)^{d-2}}
\end{equation*}
when $d \ge 3$ and that
\begin{equation*}
\begin{aligned}
 N(\sigma) = \frac{\omega_1}{2\pi}\Vol_1(\del\Omega) \sigma + \bigo{1}
\end{aligned}
 \end{equation*}
when $d = 2$.

We can now give explicit expressions for the constants $G_{p,q}$:
\begin{equation*}
 \begin{aligned}
  G_{p,q} &= \int_0^{\pi/2} \dotso \int_0^{\pi/2} \arccot\left(\frac{1}{\sqrt q}\left[1 + \sum_{j=1}^{p-1}\cot^2 \theta_j \prod_{i>j} \csc^2 \theta_i\right]^{1/2}\right) \prod_{k=1}^{p-1} \sin^k(\theta_k) \de \theta_1 \dotso \de \theta_{p-1} \\
  &= \int_0^{\pi/2} \dotso \int_0^{\pi/2} \arccot\left(\frac{1}{\sqrt q}\prod_{j=1}^{p-1} \csc \theta_j \right) \prod_{k=1}^{p-1} \sin^k(\theta_k) \de \theta_1 \dotso \de \theta_{p-1}.
 \end{aligned}
\end{equation*}
In particular, calculating the integrals for $q = 1, p = 2$ and $q = 1, p = 3$, we get: 
\begin{equation*}
 \begin{aligned}
  G_{2,1} &= \frac{1}{2}\left(-1 +\sqrt 2\right) \pi\\
  G_{3,1} &= \frac{1}{8}\left(- 2 + \pi\right) \pi
 \end{aligned}
\end{equation*}

This concludes the proof of Theorem \ref{thm:main}.

\section{Further results}
\label{other}
\subsection{Concentration of eigenfunctions}
\label{concent}
In this section, we discuss the behaviour of the eigenfunctions, more
precisely how they scar on the lower-dimensional  facets of a
cuboid. This is made precise in the following theorem, where we will slightly abuse notation and denote by $u_k$ both a Steklov eigenfunction and its boundary trace.

\begin{thm}
  \label{thm:concent}
  Let $\Omega\subset\R^d$ be the cuboid with parameters
  $a_1,\dotsc,a_d>0$. Let $p\in \set{1,\dotsc,d-1}$ and let $\tau \in \CT_p$.
  Consider the set
  $$X_\tau = \{x=(\bx_{\tau_1},\bx_{\tau_2})\in\del\Omega\,:\,x_j = \pm a_j
  \mbox{ for } j \in \tau_2\}.$$
  Then, there exists a sequence of $L^2(\del\Omega)$-normalised
eigenfunctions $\set{u_k}$ concentrating  on $X_\tau$ and
getting equidistributed  around $X_\tau$ in the following sense: for each measurable
$U\subset X_\tau$ and every $\epsilon > 0$, consider the set 
\[
U_\epsilon = \{\bx = (\bx_{\tau_1},\bx_{\tau_2}) \in \del \Omega :
\bx_{\tau_1} \in U \mbox{ and } \operatorname{dist}(\bx,U) < \epsilon\}.
\]
Then, for every $\epsilon > 0$,
\[
\lim_{k \to \infty} \int_{U_\epsilon} |u_k(\bx)|^2 dx =  \frac{\Vol_{p} (U)}{\Vol_p(X_\tau)}.
\]
\end{thm}
For example, on a cuboid of dimension 3, the set $X_\tau$ is a union
of four parallel edges in case  $p = 1$, while for $p = 2$
it is a union of two opposite faces.

\begin{proof}
 Without loss of generality, we will suppose that $U$ is a subset of one of the connected components of $X_\tau$, say the one where $x_j = a_j$ for all $j \in \tau_2$. For $k \in \N$, let $\bk = (k,\dotsc,k)\in\R^p$ and consider the pair $(\balpha^{(\bk)},\bbeta^{(\bk)})$ satisfying the compatibility and harmonicity conditions
 \begin{equation*}
 \begin{aligned}
 \alpha_i^{(k)}\cot \left(\alpha_i^{(k)}a_i\right) &= \beta_j^{(k)} \tanh\left(\beta_j^{(k)} a_j\right) \qquad\forall i \in \tau_1, j \in \tau_2 \\
  \sum_{i \in \tau_1} \left(\alpha_i^{(k)}\right)^2 &= \sum_{j \in \tau_2} \left(\beta_j^{(k)}\right)^2
\end{aligned}
  \end{equation*}
with $\balpha^{(\bk)} \in I_{2\bk}$. Note that this corresponds to choosing $\ell(i) = 0$ for all $i \in \tau_1$ and $\ell(j) = 1$ for all $j \in \tau_2$. Since
\begin{equation*}
\begin{aligned}
  \left(\sum_{i \in \tau_1} \left(\alpha_i^{(k)}\right)^2\right)^{1/2} &= k
  \overbrace{\left(\sum_{i \in \tau_1} \left(\frac{\pi}{2a_i}\right)^2\right)^{1/2}}^{A} + \bigo 1 
= A k + \bigo 1
\end{aligned}
\end{equation*}
we have that for all $j \in \tau_2$,
\begin{equation*}
\beta_j^{(k)} = \frac{A}{\sqrt q} k + \bigo 1.
\end{equation*} 
Let $v_k(\bx)$ be the associated eigenfunction, and observe that
\begin{equation*}
\begin{aligned}
 v_k(\bx)^2 &= \prod_{i \in \tau_1} \sin^2\left(\alpha_i^{(k)} x_i\right) \prod_{j \in \tau_2} \cosh^2\left(\beta_j^{(k)} x_j\right) \\
  &= \frac{1}{2^p} \prod_{i \in \tau_1}\left( 1 - \cos\left(2\alpha_i^{(k)} x_i\right)\right) \prod_{j \in \tau_2} \cosh^2\left(\beta_j^{(k)} x_j\right), \\
  &= \frac{1}{2^p} \prod_{i \in \tau_1}\left( 1 - \cos\left(\left(\frac{\pi k }{ a_i} + \bigo 1\right)x_i\right)\right) \prod_{j \in \tau_2} \cosh^2\left(\left(\frac{A}{\sqrt q} k + \bigo 1\right) x_j\right).
 \end{aligned}
\end{equation*}
Defining the normalised eigenfunction
\begin{equation*}
 u_k = \frac{v_k}{\| v_k \|_{L^2(\del \Omega)}},
\end{equation*}
we estimate both $\|v_k\|^2 := \| v_k \|^2_{L^2(\del \Omega)}$ and $\int_{U_\epsilon} v_k(x)^2 \de x$. For $\|v_k\|^2$, we have that
\begin{equation} \label{eq:riemleb}
\begin{aligned}
 \|v_k\|^2 &= \frac{1}{2^p}\prod_{i \in \tau_1}\int_{-a_i}^{a_i}1 - \cos\left(\left(\frac{\pi k }{ a_i} + \bigo 1\right) x_i \right) \de x_i \prod_{j \in \tau_2} \int_{-a_j}^{a_j} \cosh^2\left( \beta_j x_j\right) \de x_j \\
 &= \frac{1}{2^{d}}\left(\Vol_{p}(X_\tau) + \smallo 1\right)\prod_{j \in \tau_2} \int_{-a_j}^{a_j} \cosh^2\left( \beta_j x_j\right) \de x_j
\end{aligned}
 \end{equation}
from the Riemann-Lebesgue lemma and the fact that
\begin{equation*}
 \Vol(X_\tau) = 2^q \prod_{i \in \tau_1}\int_{-a_i}^{a_i} \de x_i.
\end{equation*}
Furthermore, for all $j \in \tau_2$ we have that
\begin{equation} \label{eq:intcoshfull}
 \begin{aligned}
  \int_{-a_j}^{a_j} \cosh^2\left( \beta_j x_j\right) \de x_j &= 
  \frac 1 4 \int_{-a_j}^{a_j} e^{2\left(\frac{A}{\sqrt q} k + \bigo 1\right) x_j} + e^{-2\left(\frac{A}{\sqrt q} k + \bigo 1\right) x_j} + 2 \de x_j \\
  &= \frac{\sqrt q}{4 A k} e^{2\frac{A}{\sqrt q} k a_j}\left(1 + \smallo 1\right).
 \end{aligned}
\end{equation}
Setting $C = \frac{\sqrt q}{4A}$, equations \eqref{eq:riemleb} and \eqref{eq:intcoshfull} yield together that
\begin{equation} \label{eq:normvk}
 \|v_k\|^2 = \frac{C^q}{2^{d}k^q}\Vol_{p}(X_\tau) \left(\prod_{j \in \tau_2}  e^{2\frac{A}{\sqrt q} k a_j}\right) \left(1 + \smallo 1 \right)
\end{equation}

 We now also compute the integral of $v_k^2$ on $U_\epsilon$ where we get, in a similar fashion to \eqref{eq:riemleb} that
 \begin{equation} \label{eq:intUeps}
  \int_{U_\epsilon} v_k(x)^2 \de x = \frac{1}{2^p}\left(\Vol_{p}(U) + \smallo 1\right)\prod_{j \in \tau_2} \int_{a_j - \epsilon}^{a_j} \cosh^2\left( \beta_j x_j\right) \de x_j
 \end{equation}
We also have that
\begin{equation} \label{eq:intcoshsmall}
\int_{a_j - \epsilon}^{a_j} \cosh^2\left( \beta_j x_j\right) \de x_j =\frac{C}{2} e^{2\frac{A}{\sqrt q} k a_j}\left(1 + \smallo 1\right),
\end{equation}
where once again $C = \frac{\sqrt q}{4A}$. Together, equations \eqref{eq:intUeps} and \eqref{eq:intcoshsmall} yield
\begin{equation} \label{eq:intUeps2}
 \int_{U_\epsilon} v_k(\bx) \de \bx = \frac{C^q}{2^{d}k^q}\Vol_{p}(U) \left(\prod_{j \in \tau_2}  e^{2\frac{A}{\sqrt q} k a_j}\right) \left(1 + \smallo 1 \right).
\end{equation}

Finally, putting equations \eqref{eq:normvk} and \eqref{eq:intUeps2} together yields indeed that
\begin{equation*}
\lim_{k \to \infty} \int_{U_\epsilon} u_k(x)^2 \de x = \lim_{k \to \infty} \int_{U_\epsilon} \frac{v_k(x)^2}{\|v_k\|^2} \de x = \frac{\Vol_p(U)}{\Vol_p(X_\tau)},
\end{equation*}
concluding the proof.

 \end{proof}
\subsection{The first eigenfunction}

In this section, we investigate the lowest nonzero eigenvalue
$\sigma_1$ on the cuboid. Let us first find the form of an
eigenfunction $u$ associated with $\sigma_1$. By Courant's
nodal theorem $u$ has
exactly $2$ nodal domains. Thus, one of the factors $u_j$ will
have $2$ nodal domains on the interval $[-a_j,a_j]$ and all the
other factors only one nodal domain. In other words there is one odd factor, and all the others are positive even functions. We show the following proposition.

\begin{prop}\label{prop:firsteigen}
 Suppose that $a_1 \le \dotso \le a_d$. Then there is $\bbeta = (\beta_1,\dotsc,\beta_{d-1})$ and $\alpha_d = |\bbeta| < \frac{\pi}{2a_d}$ such that
 \[
  u(x_1,\dotsc,x_d) = \sin(\alpha_d x_d) \prod_{k = 1}^{d-1} \cosh(\beta_k x_k)
 \]
is an eigenfunction with eigenvalue $\sigma_1$.

\end{prop}

\begin{proof}
We will fist show that $u$ is a product of one sine factor and $d-1$ hyperbolic cosine factors. Suppose that one of the trigonometric factors was a cosine. Let us study the number of nodal domains of $\cos(\alpha x_j)$ on the interval $[-a_j,a_j]$. By the Steklov boundary condition we have that
\begin{equation*}
 \cos(\alpha a_j) = - \sigma \alpha \sin(\alpha a_j),
\end{equation*}
 There are three possible cases, whether $\sin(\alpha a_j)$ is equal to, greater than or smaller than $0$. Since the eigenvalue $\sigma_0 = 0$ is simple, if $\sin(\alpha a_j) = 0$ it would imply that $\cos(\alpha a_j) = 0$, which is impossible. 
 
 If $\sin(\alpha_j a_j) > 0$, we have that $\cos(\alpha a_j)$ is negative. This would imply that the function $\cos(\alpha x)$ has changed sign on $[0,a_j]$ and since it is even it will have at least two zeroes on $[-a_j, a_j]$, that is at least three nodal domains, in contradiction with Courant's nodal theorem. 
 
 Finally, if $\sin(\alpha a_j) < 0$, this implies that $\alpha a_j > \pi$, meaning that $\cos(\alpha x_j)$ has changed sign at least once on $[0,a_j]$. This implies once again that there are at least three nodal domains, completing the proof that no factor is cosine.
 
 Since there can only be one odd factor, if one is linear all the other factors are a combination of cosine and hyperbolic cosine. We just proved that none of the factors are cosine, and it is impossible for a product of linear functions with only hyperbolic cosines to respect the harmonicity condition \eqref{eq:conditiondesomme}. We therefore deduce that the only odd factor of $u$ is a sine, and by the above discussion all of the other factors are hyperbolic cosine. This implies that there exists some $1 \le j \le d$, $\alpha_j$ and $\beta_k$, $k \ne j$ such that
\begin{equation*}
 u(x_1,\dotsc,x_d) = \sin (\alpha_j x_j) \prod_{k \ne j} \cosh(\beta_k x_k),
\end{equation*}
and $\alpha_j a_j < \pi/2$. The compatibility equations \eqref{equation:compatibility} hence become
\begin{equation*}
 \begin{aligned}
 \alpha_j \cot(\alpha_j a_j) &= \beta_k \tanh(\beta_k a_k) \\
  \alpha_j^2 &= |\bbeta|^2 = \sum_{k \ne j} \beta_k^2,
 \end{aligned}
\end{equation*}
and $\sigma_1$ is any member of the first equality. We show that $\sigma_1$ is smallest when $a_j$ is the largest side, i.e. $a_j = a_d$. Suppose not. Then, there is $1 \le k \le d-1$ such that an eigenvalue associated with
\begin{equation*}
 v(x_1,\dotsc,x_d) = \sin (|\bgamma| x_j) \prod_{k \ne j} \cosh(\gamma_j x_j).
\end{equation*}
is smaller than the one associated with 
\begin{equation*}
 u(x_1,\dotsc,x_d) = \sin (|\bbeta| x_d) \prod_{k \ne d} \cosh(\beta_k x_k).
\end{equation*}

The compatibility equations imply that for all $k \ne j$ and $k \ne d$,
\begin{equation*}
 \gamma_k \tanh(\gamma_k a_k) < \beta_k \tanh (\beta_k a_k).
\end{equation*}
Since $x \tanh(ax)$ is an increasing function, we deduce that $\gamma_k \le \beta_k$ for all such $k$. However, we also have that
\begin{equation*}
 |\bgamma|\cot(|\bgamma| a_k) < |\bbeta|\cot(|\bbeta|a_d)
\end{equation*}
and since $x\cot(a x )$ is decreasing on its first period and $a_k \le a_d$, this implies that $|\bgamma|\krn>\krn|\bbeta|$. From this, we therefore have that
\begin{equation*}
 \beta_j^2 + \sum_{k \ne j,d} \beta_k^2 < \gamma_d^2 + \sum_{k \ne j,d} \gamma_k^2.
\end{equation*}
Since for all $k \ne j,d$ we have that $\gamma_k < \beta_k$, we therefore deduce that $\beta_j < \gamma_d$. However, once again using the compatibility conditions, we have that
\begin{equation*}
 \gamma_d \tanh(\gamma_d a_d) < \beta_j \tanh (\beta_j a_j).
\end{equation*}
Since $a_d > a_j$, by monotonicity of $x \tanh(a x)$ we deduce that $\gamma_d < \beta_j$, a contradiction. Hence, we have that the first eigenfunction is, taking into account that $\alpha_d = |\bbeta|$,
\begin{equation*}
 u(x_1,\dotsc,x_d) = \sin(|\bbeta| x_d) \prod_{j = 1}^{d-1} \cosh(\beta_j x_j),
\end{equation*}
\end{proof}
concluding the proof of the proposition.
\subsection{Proof of Theorem \ref{thm:isoperimetric}}
\label{proof:isoperim}

The first eigenvalue is given by the following min-max principle :
\begin{equation*}
 \sigma_1(\Omega) = \inf_{\substack{u \in C^\infty(\Omega) \\ \int_{\del \Omega} u = 0}} R_\Omega[u] = \inf_{\substack{u \in C^\infty(\Omega) \\ \int_{\del \Omega} u = 0}} \frac{\int_\Omega |\nabla u|^2}{\int_{\del \Omega} u^2}.
\end{equation*}
Denote by $\Omega_0$ the cube $[-1,1]^d$. Then, for any cuboid $\Omega = [-a_1,a_1]\times \dotso \times [-a_d,a_d]$ we have that
\begin{equation*}
 \int_\Omega f(x) \de x = \int_{\Omega_0} f\left(a_1 x_1,\dotsc,a_d x_d\right) \prod_{i = 1}^d a_i \de x
\end{equation*}
and
\begin{equation} \label{eq:rescaling}
\begin{aligned}
 \int_{\del \Omega} f(x) \de x &= \sum_{j = 1}^d \int_{\del \Omega \cap \set{x_j = \pm a_j}} f(x) \de x \\
 &= \sum_{j = 1}^d \int_{\del \Omega_0 \cap \set{x_j = \pm 1}} f\left(a_1 x_1,\dotsc,a_dx_d\right) \prod_{i \ne j} a_i \de x.
 \end{aligned}
\end{equation}
This allows us to consider integration only on $\Omega_0$ for $R_\Omega$. Observe that the eigenspace of $\sigma_1(\Omega_0)$ has dimension $d$, and that a basis for it is given by
\begin{equation*}
 u_j(x_1,\dotsc,x_d) = \sin(|\beta| x_d) \prod_{i \ne j} \cosh(\beta_i x_i).
\end{equation*}
The eigenfunctions $u_j$ are orthogonal to constants in the scalar product given by the rescaled integral \eqref{eq:rescaling}. Indeed, on all faces where the $\sin$ factor is not constant, the integral vanishes since it is an odd function. On the pair of faces where the $\sin$ factor is constant, we have that $u_j(x_1,\dotsc,a_j,\dotsc,x_d) = - u_j(x_1,\dotsc,-a_j,\dotsc,x_d)$ hence the integrals cancel out on these two faces.

Consider the eigenfunction
\begin{equation*}
 u = \sum_{j = 1}^d u_j.
\end{equation*}
It is easy to see that the integral of $u^2$ on any face of $\Omega_0$ is identical, and we have that $R_{\Omega_0}[u] = \sigma_1(\Omega_0)$. We now compute
\begin{equation*}
 \begin{aligned}
\frac{1}{R_{\Omega}[u]} &= \frac{\sum_{j = 1}^d \prod_{i \ne j} a_i \int_{\del \Omega_0 \cap \set{x_j = \pm 1}} u^2 \de x}{\prod_{j=1}^{d} a_j \int_{\Omega_0} |\nabla u|^2 \de x} \\
  &= \frac{1}{R_{\Omega_0}[u]}\frac{\frac 1 d\sum_{j = 1}^d \prod_{i \ne j} a_i }{\prod_{j =1 }^d a_j} .
 \end{aligned}
\end{equation*}
Fix the volume $\Vol_d(\Omega) = \Vol_d(\Omega_0)$, hence $\prod_{j} a_j = 1$. Then, from the inequality of arithmetic and geometric means, 
\begin{equation*}
\begin{aligned}
 \frac{R_{\Omega_0}[u]}{R_\Omega[u]} = \frac 1 d \sum_{j=1}^d \prod_{i \ne j} a_i 
 \ge \left(\prod_{j=1}^d a_j^{d-1}\right)^{1/d} 
 = 1,
\end{aligned}
 \end{equation*}
with equality if and only if for all $j,k$, $\prod_{i \ne j} a_i = \prod_{i \ne k} a_i$, which is true if and only if $a_j = a_k$ for all $j,k$, which implies in turn that $\sigma_1(\Omega) \le \sigma_1(\Omega_0)$, with equality if and only if $\Omega$ is a cube.

On the other hand, fix the area, $\Vol_{d-1}(\Omega) = \Vol_{d-1}(\Omega_0)$, hence $\sum_{j} \prod_{i \ne j} a_i = d$. Then,
\begin{equation*}
\begin{aligned}
 \frac{R_{\Omega_0}[u]}{R_\Omega[u]} = \left( \prod_{j} a_j\right)^{-1} = \left(\prod_{j=1}^d \prod_{i \ne j} a_i\right)^{\frac{d(1-d)}{d}} \ge \left(\frac{1}{d}\sum_{j=1}^d \prod_{i \ne j} a_i\right)^{\frac{1-d}{d}} = 1,
\end{aligned}
 \end{equation*}
with equality in the same case as before. Once again, this implies that $\sigma_1(\Omega) \le \sigma_1(\Omega_0)$, with equality if and only if $\Omega$ is a cube.

\subsection{Proof of Corollary \ref{cor:rect}} 
\label{section:spectraldetermination}

We want to show that
among all rectangles, the Steklov spectrum determines the lengths
$a_1,a_2$ of its sides. From spectral asymptotics, the perimeter of
the rectangle is obtained, giving $L = a_1 + a_2$, supposing without loss of generality that $a_1 \le
a_2$. On the other hand, we have $\sigma_1$, and we know that it is
the smallest root of 
\begin{equation*}
 \sigma_1 = \alpha \cot(\alpha a_1) = \alpha \tanh(\alpha a_2).
\end{equation*}
Rewriting these to yield $a_2$ as a function of $\alpha,$ $L$ and $\sigma_1$ gives
\begin{equation}\label{eq:curve1}
 a_2 = f(\alpha) = \frac{1}{\alpha} \arctanh\left(\frac{\sigma_1}{\alpha}\right)
\end{equation}
and
\begin{equation}\label{eq:curve2}
 a_2 = g(\alpha) =  L - \frac{1}{\alpha} \arccot\left(\frac{\sigma_1}{\alpha}\right).
\end{equation}
Given $\sigma_1$ and $L$, the intersection of these curves yield possible values $a_2$ for $\alpha$. We now show that they intersect at only one point. Equation \eqref{eq:curve1} is defined for $\alpha > \sigma_1$ and taking the derivative yields
\begin{equation}
 f'(\alpha) = -\frac{\arctanh\left(\frac{\sigma_1}{\alpha}\right)}{\alpha^2}-\frac{\sigma_1}{\alpha^3 \left(1-\frac{\sigma_1^2}{\alpha^2}\right)},
\end{equation}
which is always negative for $\alpha > \sigma_1$, hence $f$ is decreasing. We now show that $g$ is increasing on $[\sigma_1,\infty)$. We have that
\begin{equation*}
 g'(\alpha) = \frac{\arccot\left(\frac{\sigma_1 }{\alpha }\right)}{\alpha ^2}-\frac{\sigma_1 }{\alpha ^3 \left(1+\frac{\sigma_1^2}{\alpha ^2}\right)}.
\end{equation*}
Thus, $g'$ is positive if
\begin{equation*}
 \alpha \arccot\left(\frac{\sigma_1 }{\alpha }\right) \left(1+\frac{\sigma_1^2}{\alpha ^2}\right) - \sigma_1 \ge 0.
\end{equation*}
However, we have that
\begin{equation*}
\begin{aligned}
 \alpha \arccot\left(\frac{\sigma_1 }{\alpha }\right) \left(1+\frac{\sigma_1^2}{\alpha ^2}\right) - \sigma_1  \ge \frac \pi 4 \alpha + \frac \pi 4 \frac{\sigma_1^2}{\alpha^2} - \sigma_1
 \end{aligned}
\end{equation*}
hence we need to have that $\alpha^2 - \frac{4\sigma_1 \alpha}{\pi} + \sigma_1^2 \ge 0$. This quantity is positive at $\alpha = \sigma_1$ since $2 \ge 4/\pi$ and it is increasing since
\begin{equation*}
 2\alpha > \frac{4 \sigma_1 }{\pi} 
\end{equation*}
for $\alpha \ge \sigma_1$. We conclude that $g$ is increasing. This implies that $f$ and $g$ have
exactly one intersection point, say at $\alpha_0$. We have that $a_2 =
f(\alpha_0) = g(\alpha_0)$ and $a_1 = L - a_2$. Note that since the
square maximises $\sigma_1$ and since the eigenvalues are
continuous functions of the side lengths of a rectangle this means
that among all rectangles with given area or perimeter, $\sigma_1$ is
a decreasing function of $a_2$.

\appendix

\section{Proof of Lemma \ref{lem:arccot}}
\label{appendix:arccot}
\begin{lem} \label{lem:arccot}
Let 
\begin{equation*}
 f_i(\bx) = \arccot\left(c\left[1 + \sum_{j \ne i} \left(\frac{x_j}{x_i}\right)^2\right]^{1/2}\right).
\end{equation*}
 for some $c > 0$ and where by convention $\arccot(\infty) = 0$,  and let $\psi : \R^p \to \R^p$ be a bounded function. Then,
 \begin{equation} \label{eq:arccotdiff} \tag{A.1}
  \left|f_i(\bx + \psi(\bx)) - f_i(\bx)\right| = \bigo{|\bx|^{-1}}.
 \end{equation}
\end{lem}

\begin{proof}
We have that
\begin{equation*}
 \left|f_i(\bx) - f_i(\bx_0)\right| = \bigo{|\bx - \bx_0||\nabla f(\bx_0)|}. 
\end{equation*}

Consider spherical coordinates $(r,\theta_1,\dotsc,\theta_{p-1})$
 \begin{equation*}
  \begin{aligned}
   r &= |\bx|, \\
   x_j &= r \cos(\theta_j) \prod_{i < j} \sin(\theta_i),
  \end{aligned}
 \end{equation*}
where by convention $\theta_p = 0$. 

Denote $\bx = (r,\btheta)$ and $\bx + \psi(\bx) = (r_\psi,\btheta_\psi)$. It is clear that since $\psi$ is bounded we have that
\begin{equation*}
 |\btheta - \btheta_\psi| = \bigo{r^{-1}}.
\end{equation*}
Indeed, from planar geometry we get that
\begin{equation*}
 \tan(|\btheta - \btheta_\psi|) \le \frac{\sup_{\bx \in \R^p} \psi(\bx)}{r}.
\end{equation*}

One can observe that the functions in Equation \eqref{eq:arccotdiff} depend only on $\btheta_\psi$ and $\btheta$. Hence, showing that the gradient is bounded in $\btheta$ implies that $|f_i(\bx + \psi(\bx)) - f_i(\bx)| = \bigo{r^{-1}}$.

By symmetry, we can suppose without loss of generality that $i = p$ in Equation \eqref{eq:arccotdiff}. Then, using repeatedly the identity $1 + \cot^2 \theta = \csc^2\theta$ we have that
\begin{equation*}
 f_p(\bx) = \arccot\left(c\prod_{j=1}^{p-1} \csc\theta_j\right).
\end{equation*}
Now, we have that
\begin{equation*}
 \del_{\theta_j} f_p(x) = c\frac{\cot \theta_j \prod_{k=1}^{p-1} \csc \theta_k}{1 + c^2\prod_{k=1}^{p-1} \csc^2 \theta_k}.
\end{equation*}
This is bounded since when $\theta_j \to n \pi$, the singularities are of the same order on the numerator and denominator while when it is any other $\theta_i \to n \pi$, the singularities are of order $1$ in the numerator and $2$ in the denominator. This concludes the proof.
\end{proof}

\section{Positivity of the constant $C_2$} \label{appendix:positive}

We can rewrite $C_2$ as
\[
 C_2 = \frac{ (2^{\frac{d+2}{2}} - 2) \pi \omega_{d-2} - 2^{d+1} G_{d-1,1}}{2(2\pi)^{d-1}},
\]
and we need to show that $C_2 > 0$ for $d \ge 3$. This will be done by showing that
\begin{equation} \label{eq:greaterthanone} \tag{B.1}
 \frac{(2^{\frac{d+2}{2}} - 2) \pi \omega_{d-2}}{2^{d+1} G_{d-1,1}} > 1.
\end{equation}
Let us first observe that the integrand in $G_{d-1,1}$ is positive and that for any $\btheta \in [0,\pi/2]^{d-2}$, we have that
\[
 \arccot\left(\prod_{j = 1}^{d-2} \csc \theta_j\right) \le \arccot(1) < 1.
\]
Hence,
\begin{align*}
 G_{d-1,1} &= \int_0^{\pi/2} \dotso \int_0^{\pi/2}
 \arccot\left(\prod_{j=1}^{d-2} \csc \theta_j \right)
 \prod_{k=1}^{d-2} \sin^k(\theta_k) \de \theta_1 \ldots \de \theta_{d-2} \\
 &\le \prod_{k=1}^{d-2} \int_0^{\pi/2}  \sin^k(\theta_k) \de \theta_k \\&= \frac{2^{2-d}\pi^{\frac{2 - d}{2}}}{\Gamma\left(\frac d 2\right)}.
\end{align*}
The last equality is true for $d = 3$, and is seen to be true for all $ d \ge 3$ by induction using the identity \cite[3.621 (1)]{gradshteynryzhik}
\[
 \int_0^{\pi/2} \sin^k(\theta) \de \theta = 2^{k - 1} B\left(\frac{k+1}{2},\frac{k+1}{2}\right).
\]
and the Gamma function duplication identity
\[
 \Gamma(\mu)\Gamma\left(\mu + 1/2\right) = 2^{1-2\mu} \sqrt \pi \Gamma(2\mu).
\]
Using the fact that
\[
 \omega_{d-2} = \frac{\pi^{\frac{d-2}{2}}}{\Gamma\left(\frac d 2\right)}
\]
and replacing in Equation \eqref{eq:greaterthanone} we have that
\begin{align*}
 \frac{(2^{\frac{d+2}{2}} - 2) \pi \omega_{d-2}}{2^{d+1} G_{d-1,1}} &\ge \frac{(2^{\frac{d+2}{2}} - 2) \pi}{8} > 1
\end{align*}
for all $d \ge 3$, concluding the proof that $C_2 > 0$.

\bibliographystyle{plain}
\bibliography{biblio}

\def\cprime{$'$} \def\cprime{$'$} \def\cprime{$'$} \def\cprime{$'$}
\begin{thebibliography}{10}

\bibitem{Agranovich}
M.~S. Agranovich.
\newblock On a mixed {P}oincar\'e-{S}teklov type spectral problem in a
  {L}ipschitz domain.
\newblock {\em Russ. J. Math. Phys.}, 13(3):239--244, 2006.

\bibitem{AuchmutyCho}
G.~Auchmuty and M.~Cho.
\newblock Boundary integrals and approximations of harmonic functions.
\newblock {\em Numer. Funct. Anal. Optim.}, 36(6):687--703, 2015.

\bibitem{brock}
F.~Brock.
\newblock An isoperimetric inequality for eigenvalues of the {S}tekloff
  problem.
\newblock {\em Z. Angew. Math. Mech.}, 81(1):69--71, 2001.

\bibitem{BFNT}
D.~{Bucur}, V.~{Ferone}, C.~{Nitsch}, and C.~{Trombetti}.
\newblock {Weinstock inequality in higher dimensions}.
\newblock {\em ArXiv e-print: 1710.04587}, October 2017.

\bibitem{girpoltJST}
A.~Girouard and I.~Polterovich.
\newblock Spectral geometry of the {S}teklov problem ({S}urvey article).
\newblock {\em J. Spectr. Theory}, 7(2):321--359, 2017.

\bibitem{GL}
K.~{Gittins} and S.~{Larson}.
\newblock {Asymptotic behaviour of cuboids optimising Laplacian eigenvalues}.
\newblock {\em ArXiv e-print: 1703.10249}, March 2017.

\bibitem{gradshteynryzhik}
I.~S. Gradshteyn and I.~M. Ryzhik.
\newblock {\em Table of integrals, series and products}, volume~7.
\newblock Academic Press, Burlington, 2007.

\bibitem{hlawka}
E.~Hlawka.
\newblock \"uber {I}ntegrale auf konvexen {K}\"orpern. {I}.
\newblock {\em Monatsh. Math.}, 54:1--36, 1950.

\bibitem{IL}
A.~Iosevich and E.~Liflyand.
\newblock {\em Decay of the {F}ourier transform, analytic and geometric
  aspects}.
\newblock Birkh\"auser/Springer, Basel, 2014.

\bibitem{Iv}
V.~Ivrii.
\newblock Private communication.
\newblock 2017.

\bibitem{lagaceparnovski}
J.~{Lagacé} and L.~{Parnovski}.
\newblock A generalised {G}auss circle problem and integrated density of
  states.
\newblock {\em J. Spectr. Theory}, 6(4):859--879, 2016.

\bibitem{LPPS}
M.~{Levitin}, L.~{Parnovski}, I.~{Polterovich}, and D.~A. {Sher}.
\newblock Sloshing, {S}teklov and corners {I}: Asymptotics of sloshing
  eigenvalues.
\newblock {\em ArXiv e-print: 1709.01891}, September 2017.

\bibitem{pinascorossi}
J.~P. Pinasco and J.~D. Rossi.
\newblock Asymptotics of the spectral function for the {S}teklov problem in a
  family of sets with fractal boundaries.
\newblock {\em Appl. Math. E-Notes}, 5:138--146, 2005.

\bibitem{PS}
I.~Polterovich and D.~A. Sher.
\newblock Heat invariants of the {S}teklov problem.
\newblock {\em J. Geom. Anal.}, 25(2):924--950, 2015.

\bibitem{randol}
B.~Randol.
\newblock A lattice-point problem.
\newblock {\em Trans. Amer. Math. Soc.}, 121:257--268, 1966.

\bibitem{Strauss}
W.~A. Strauss.
\newblock {\em Partial differential equations, an introduction}.
\newblock John Wiley \& Sons, Inc., New York, 1992.

\bibitem{ArnoldTan}
A.~{Tan}.
\newblock {The Steklov Problem on Rectangles and Cuboids}.
\newblock {\em ArXiv e-print: 1711.00819}, November 2017.

\bibitem{vBG}
M.~van~den Berg and K.~Gittins.
\newblock Minimizing {D}irichlet eigenvalues on cuboids of unit measure.
\newblock {\em Mathematika}, 63(2):469--482, 2017.

\bibitem{wein}
R.~Weinstock.
\newblock Inequalities for a classical eigenvalue problem.
\newblock {\em J. Rational Mech. Anal.}, 3:745--753, 1954.

\end{thebibliography}

\end{document}